%% file: wpi_arxiv.tex
\documentclass[10pt]{article}

\input{_macros}
\usepackage{eqparbox}
\usepackage[numbers]{natbib}

\usepackage[dvipsnames]{xcolor}

\newcommand{\D}{\mathrm{d}}
\newcommand{\R}{\mathbb{R}}
\newcommand{\N}{\mathbb{N}}
\newcommand{\E}{\mathbb{E}}

\newcommand{\FIc}{\beta}
\newcommand{\var}{\operatorname{Var}}

\newcommand{\osc}{\operatorname{Osc}}

\newcommand{\wass}{\operatorname{W}}

\newcommand\mmid{\mathbin{\|}}
\def\abs#1{\left| #1 \right|}
\def\norm#1{\left\|{#1}\right\|}
\def\binner#1#2{\left\langle{#1}, {#2}\right\rangle}

\def\Epi#1{\E_\pi\big[{#1}\big]}
\def\Earg#1{\E\left[{#1}\right]}
\def\WPI{\mathsf{WPI}}

\def\PI{\mathsf{PI}}
\def\LOI{\mathsf{LOI}}

\def\reals{\mathbb{R}}
\def\KL{\mathsf{KL}}
\def\ccplx#1#2{\mathscr{C}^{\text{LMC}}_{R_{#1}, R_{#2}}}
\def\tcplx#1#2{\mathscr{C}^{\text{LD}}_{R_{#1}, R_{#2}}}

\hypersetup{
  colorlinks,
  linkcolor={red!50!black},
  citecolor={blue!50!black},
  urlcolor={blue!80!black}
}

\begin{document}
\title{Towards a Complete Analysis of Langevin Monte Carlo: \\ Beyond Poincar\'e Inequality}
\author{
Alireza Mousavi-Hosseini\thanks{Both authors contributed equally to this work.} \thanks{Department of Computer Science at University of Toronto, and Vector Institute. \texttt{mousavi@cs.toronto.edu}}
\and Tyler Farghly\footnotemark[1] \thanks{Department of Statistics, University of Oxford, UK. \texttt{farghly@stats.ox.ac.uk}}
\and Ye He\thanks{Department of Mathematics, University of California, Davis. \texttt{leohe@ucdavis.edu}}
\and Krishnakumar Balasubramanian\thanks{Department of Statistics, University of California, Davis. \texttt{kbala@ucdavis.edu}}
\and Murat A. Erdogdu\thanks{Department of Computer Science and Department of Statistical Sciences at University of Toronto, and Vector Institute. \texttt{erdogdu@cs.toronto.edu}}
}

\maketitle

\begin{abstract}
Langevin diffusions are rapidly convergent under appropriate functional inequality assumptions. Hence, it is natural to expect that with additional smoothness conditions to handle the discretization errors, their discretizations like the Langevin Monte Carlo (LMC) converge in a similar fashion. This research program was initiated by~\cite{Vempala2019-jz}, who established results under log-Sobolev inequalities.~\cite{Chewi2021-vj} extended the results to handle the case of Poincar\'e inequalities. In this paper, we go beyond Poincar\'e inequalities, and push this research program to its limit. We do so by establishing upper and lower bounds for Langevin diffusions and LMC under weak Poincar\'e inequalities that are satisfied by a large class of densities including polynomially-decaying heavy-tailed densities (i.e., Cauchy-type). Our results explicitly quantify the effect of the initializer on the performance of the LMC algorithm. In particular, we show that as the tail goes from sub-Gaussian, to sub-exponential, and finally to Cauchy-like, the dependency on the initial error goes from being logarithmic, to polynomial, and then finally to being exponential. This three-step phase transition is in particular unavoidable as demonstrated by our lower bounds, clearly defining the boundaries of LMC.
\end{abstract}

\section{Introduction}
Consider the problem of sampling from a target probability density $\pi \propto \exp(-V)$ on $\R^d$ using the canonical algorithm, Langevin Monte Carlo (LMC). The LMC iterations are given by
\begin{equation}\label{eq:lmc}
    x_{k+1} = x_k - h\nabla V(x_k) + \sqrt{2h}\xi_k, \tag{LMC}
\end{equation}
where $h> 0$ is the step size, and $(\xi_k)_{k \in \N}$ is an i.i.d.\ sequence of standard Gaussian random vectors. This algorithm is based on discretizing the following stochastic differential equation (SDE), often referred to as the (overdamped) Langevin diffusion,
\begin{equation}\label{eq:Langevin}
    \D X_t = -\nabla V(X_t)\D t + \sqrt{2} \, \D B_t, \tag{LD}
\end{equation}
where $(B_t)_{t \in \R_+}$ is the $d$-dimensional standard Brownian motion. When $\pi$ is (strongly) log-concave and smooth, non-asymptotic convergence of LMC has been extensively studied~\citep{dalalyan2012sparse, dalalyan2017theoretical, dalalyan2017further, durmus2019high, durmus2019analysis}.

The Langevin diffusion~\eqref{eq:Langevin}, however, converges under relatively milder functional inequality assumptions
which are less restrictive
compared to global curvature conditions like log-concavity. Indeed, while log-concavity restricts $\pi$ to be uni-modal, functional inequality based conditions allow for some degree of multi-modality in $\pi$  \citep{chen2021dimension}. Furthermore, functional inequalities  characterize a wide range of target densities by capturing the tail behavior of the potential. For example, a target potential with tail growth $V(x) \approx \|x\|^\alpha$
at infinity, would satisfy a logarithmic Sobolev inequality (LSI) when $\alpha=2$, and satisfies a Poincar\'e inequality (PI) when $\alpha=1$. Thus, an LSI induces a faster tail growth and is consequently  a stronger condition than a PI.

Motivated by this, the following research program was initiated by \cite{Vempala2019-jz}: \emph{Can one provide convergence guarantees for \eqref{eq:lmc} when the target density satisfies a functional inequality and a smoothness condition?} The authors answered the question in the affirmative, showing that the following two conditions on the target $\pi\propto e^{-V}$ suffice to establish a sharp non-asymptotic guarantee for LMC: $(\mathrm{i})$ $\pi$ satisfies an LSI and $(\mathrm{ii})$ $\nabla V$ is Lipschitz continuous.
\cite{Chewi2021-vj} extended this framework significantly; among other contributions, they also proved that LSI can be replaced with a Lata\l{}a-Oleszkiewicz inequality (LOI), which can cover a range of tail behavior, i.e.\ $\alpha \in [1,2]$, interpolating between the edge cases PI ($\alpha=1$) and LSI ($\alpha=2$). These works provide a thorough characterization of the convergence of LMC for at least linearly growing potentials, and to our knowledge, providing the state of the art guarantees under minimal set of conditions for this algorithm. However, it is rather unclear how much further this program can be extended. For example, what is the threshold for the tail behavior $\alpha$ beyond which LMC \emph{fails}, if at all it fails to sample from such \emph{heavy-tailed} targets?

In this paper, we aim to complete the program initiated by \cite{Vempala2019-jz}, and push the convergence analysis of LMC to its limits. We study the behavior of LMC for potentials that satisfy a family of weak-Poincar\'e inequalities (WPI), which are one of the mildest conditions required to prove the ergodicity of the Langevin diffusion~\citep{Rockner2001-zs,bakry2014analysis}. A particularly interesting aspect of WPI is that virtually \emph{any target density} satisfies such an inequality. Thus, by proving a convergence guarantee for LMC under a WPI with explicit rate estimates, we establish its convergence universally for any sufficiently smooth target. Interestingly, for targets with sublinear tails, i.e.\
$V(x) \approx \|x\|^\alpha$ for $\alpha \in(0,1)$, our rate is polynomial in the initial error,
and smoothly extrapolates the rate derived by \cite{Chewi2021-vj} which was originally covering the regime $\alpha\in[1,2]$.
In the case $\alpha$ approaches 0, however, the tail is logarithmic $V(x) \approx \ln(\|x\|)$, e.g.\ for Cauchy-type distributions, and our rate estimates exhibit an exponential dependence on the initial error; thus, when there is no warm-start available,
LMC would require exponentially many iterations in the initial error for such targets with extreme heavy tails. We also provide a lower bound for LMC under a general tail-growth condition, proving that for Cauchy-type distributions, there is an initialization such that this exponential dependence is unavoidable.
Our main contributions can be summarized as follows.
\begin{itemize}[leftmargin=.2in]
    \item For a target $\pi\propto e^{-V}$ satisfying a WPI with a H\"older continuous $\nabla V$, we establish non-asymptotic convergence guarantees for LMC and the Langevin diffusion in R\'enyi divergence. Since any distribution with a locally bounded potential $V$ satisfies a WPI \citep{Rockner2001-zs}, our results provide a convergence guarantee for LMC for any sufficiently smooth target.

    \item We prove WPIs with explicit dimension dependence for two model examples of heavy-tailed distributions that do not satisfy a Poincar\'e inequality, hence cannot be covered by the results of \cite{Chewi2021-vj}. First, we consider sub-linearly decaying potentials of the form $V(x) = (1 + \norm{x}^2)^{\alpha/2}$ and establish a rate, which coincides with the estimates of \cite{Chewi2021-vj}, but also holds for $\alpha\in(0,1)$, namely beyond a Poincar\'e inequality. Notably, this rate is polynomial in the initial error for all $\alpha>0$. We also consider the case of extreme heavy tails, i.e.\
    Cauchy-type potentials with $\nu > 0$ degrees of freedom of the form $V(x) = \tfrac{d + \nu}{2}\ln(1 + \norm{x}^2)$, which does not have moments of order $\geq \nu$ defined. For this class of distributions, we prove that, even though LMC converges in R\'enyi divergence of any order, the dependence on the initial error may be exponential, which may limit its performance severely.

    \item Finally, we establish lower bounds for the complexity of LMC as well as the Langevin diffusion in R\'enyi divergence, under various tail growths in the range $\alpha \in [0,2]$.
    Our lower bounds indicate that,
    as the tail growth becomes heavier, LMC and the diffusion both exhibit a slow start behavior by having
a worse dependence on the initial divergence.
In the particular case of Cauchy-type targets,  the exponential dependence on the initial error for LMC and the diffusion is unavoidable, unless there is a good initialization available.
\end{itemize}
\noindent \textbf{More related work.}~There have been innumerable works in the recent past focusing on (strongly) log-concave sampling with LMC, which makes it hard to summarize them here. We refer the interested reader to~\cite{chewisamplingbook} for a detailed exposition.
Beyond the log-concave setting, the assumption of dissipativity, which controls the growth order of the potential, is used in a large number of prior works to obtain convergence rates for LMC \citep{Durmus2017-ww,raginsky2017non, erdogdu2018global, erdogdu2021convergence, mou2022improved, erdogdu2022convergence}. Additionally, a recent result by~\cite{balasubramanian2022towards} characterized the performance of (averaged) LMC for target densities that are only H\"{o}lder continuous (without any further functional inequality or curvature-based assumptions); however, their guarantees were provided in the relatively weaker Fisher information.

We also remark that in the (strongly) log-concave or light-tailed settings, several non-asymptotic results exist on variants of LMC, including higher order integrators~\citep{shen2019randomized, li2019stochastic,he2020ergodicity}, the underdamped Langevin Monte Carlo \citep{cheng2018underdamped, eberle2019couplings, cao2019explicit, dalalyan2020sampling}, and the Metropolis-adjusted Langevin Algorithm \citep{dwivedi2018log, lee2020logsmooth, chewi2021optimal, wu2021minimax}.

Research on the analysis of heavy-tailed sampling is relatively  scarce, especially results that are non-asymptotic in nature. \cite{chandrasekaran2009sampling} studied the iteration complexity of the Metropolis random walk algorithm for sampling from s-concave distributions. \cite{Jarner2007-du} established polynomial ergodicity results for several sampling algorithms including LMC. \cite{Kamatani2014-ri} developed modifications of the standard Metropolis random walk that are suitable for handling heavy-tailed targets and established asymptotic convergence results. Recently, \cite{andrieu2022poincar} analyzed Metropolis random walk algorithms under WPI and established rate of convergence results in variance-like metrics. \cite{johnson2012variable} introduced a variable transformation method for Metropolis random walk algorithms, transforming heavy-tailed densities into light-tailed ones using invertible transformations to benefit from existing light-tailed sampling algorithms. \cite{He2022-np} looked into ULA on a class of transformed densities and provided non-asymptotic results, mainly focusing on isotropic densities. The transformation approach has been extended in recent works such as~\cite{yang2022stereographic}, and has also been used to prove asymptotic exponential ergodicity for various sampling algorithms in the heavy-tailed settings~\citep{deligiannidis2019exponential,durmus2020geometric,bierkens2019ergodicity}.

While the literature on upper bounds on the complexity of sampling algorithms has seen significant progress, the literature on lower bounds is quite limited. Algorithm-independent query complexity of sampling from strongly log-concave distributions in one dimension was obtained by \cite{Chewi2022-qd}. \cite{Li2021-rc} established lower bounds for LMC for sampling from  strongly log-concave distributions.~\cite{chatterji2022oracle} established lower bounds for sampling from strongly log-concave distributions in the stochastic setting, when the gradients are observed with noise.~\cite{ge2020estimating} established lower bounds for the related problem of estimating the normalizing constants of a log-concave density. \cite{lee2021lower} and \cite{wu2021minimax} established lower bounds for the class of metropolized algorithms (including metropolized Langevin and Hamiltonian Monte Carlo methods) for sampling from strongly log-concave distributions. Finally, lower bounds in Fisher information for non-log-concave sampling were obtained in~\cite{Chewi2022-fv}.

\vspace{0.1in}

\noindent \textbf{Notation.}
Throughout the paper, we will use $\pi \propto \exp(-V)$ to denote the target probability measure with unnormalized potential $V$, and $\rho_t$ and $\mu_k$ to denote the law of the Langevin diffusion at time $t$, and the law of the LMC at iteration $k$. $(P_t)_{t \in \R_+}$ will denote the Markov semigroup of the Langevin diffusion. Probability measures we work with in this paper admit densities with respect to the Lebesgue measure, and we will use the same notation for their densities. We will use $\tilde{\Theta}_\Psi$ and $\tilde{O}_\Psi$ to hide polylog factors and constants depending only on the set of variables $\Psi$. $\Gamma(z) \coloneqq \int_0^\infty t^{z-1}e^{-t}\D t$ for $z > 0$ denotes the Gamma function, and $\omega_d$ is the volume of the unit $d$-ball.

\section{Weak Poincar\'e Inequalities and R\'enyi Convergence of the Diffusion}\label{sec:diff}
We consider a class of functional inequalities introduced by \cite{Rockner2001-zs}, motivated by the work of~\cite{liggett1991}\footnote{We refer the interested reader to~\cite{aida1998uniform} and~\cite{mathieu1998quand} for other attempts to propose weaker versions of Poincar\'e inequalities, and to~\cite{Rockner2001-zs} for their relationship with Definition~\ref{def:wpi}.}. Throughout this work, we avoid concerns regarding the domain of the generator for the diffusion \eqref{eq:Langevin} by assuming that the set of infinitely differentiable functions, $\mathcal{C}^\infty(\R^d)$, forms a core for the domain. For example, this is given when $V$ is infinitely differentiable itself (see, \emph{e.g.,} \cite[Proposition 3.2.1]{bakry2014analysis}).
\begin{definition}[Weak Poincar\'{e} Inequality]\label{def:wpi}
A probability measure $\pi$ on $\R^d$ satisfies a weak Poincar\'e inequality (WPI) if there exists non-increasing $\FIc_\WPI : (0,\infty) \to \R_+$ and $\Phi : L^2(\pi) \to [0,\infty]$ with $\Phi(cf + a) = c^2\Phi(f)$ for every $c,a \in \R$ and $f \in L^2(\pi)$, such that for every $f \in \mathcal{C}^\infty(\R^d)$,
\begin{equation}
    \var_\pi(f) \leq \FIc_\WPI(r)\ \E_\pi \big[{\norm{\nabla f}^2}\big] + r\, \Phi(f), \quad \forall r > 0.\label{eq:wpi}
    \tag{WPI}
\end{equation}
\end{definition}

We remark that virtually any target measure of interest in sampling satisfies such an inequality. More specifically, \cite{Rockner2001-zs} showed that $\pi \propto e^{-V}$ satisfies a WPI with $\Phi(\cdot) = \osc(\cdot)^2\coloneqq (\sup f - \inf f)^2$ and for some $\FIc_\WPI$ as soon as $V$ is locally bounded. In the special cases where $\Phi = 0$ or the function $\FIc_\WPI$ is uniformly bounded, the above inequality reduces to the classical Poincar\'e inequality which reads, for a constant $\FIc_\PI$ and for every $f \in \mathcal{C}^\infty(\R^d)$,
\begin{equation}
\var_\pi(f) \leq \FIc_\PI\ \E_\pi \big[{\norm{\nabla f}^2}\big] .\label{eq:pi}
    \tag{PI}
\end{equation}
The tail properties of the distribution $\pi$ are captured by the function $\FIc_\WPI$, which will essentially determine the convergence rate of LMC.
We will present our convergence guarantees under the generic condition \eqref{eq:wpi}, and for several model examples, we will derive explicit estimates of $\FIc_\WPI$ to make our results more explicit.

Functional inequalities of the form \eqref{eq:wpi} naturally extend Poincar\'e inequalities \eqref{eq:pi} to arbitrary distributions, removing any tail growth requirements. In particular, as PI is equivalent to an exponential $L^2$ convergence rate for the Markov semigroup, a WPI is equivalent to a subexponential $L^2$ convergence rate \citep{Rockner2001-zs,bakry2014analysis}. %
Similarly, one can also replace the variance term in \eqref{eq:wpi} with entropy, in which case the functional inequality is of the form of a weak log-Sobolev inequality (WLSI). As shown by \cite{Cattiaux2007-gy}, a WLSI is equivalent to a WPI; thus, we find it sufficient to present our results in terms of the WPI.

Following recent works (see, \emph{e.g.,} \cite{ganesh2020faster, erdogdu2022convergence, Chewi2021-vj}), we use R\'enyi divergence as a measure of distance between two probability distributions. R\'enyi divergence of order $q$ is defined by
\begin{equation}
     R_q(\rho\mmid\pi) \coloneqq
     \frac{1}{q-1}\ln \norm{\frac{\D\rho}{\D\pi}}^q_{L^q(\pi)} \ \  \text{ for }\ \  1< q < \infty,
\end{equation}
when $\rho$ is absolutely continuous with respect to $\pi$, and $+\infty$ otherwise. By Jensen's inequality, $R_q(\rho\mmid\pi)$ is non-decreasing in $q$. If we consider the limits, $(i)$ as $q\downarrow 1$ it reduces to KL divergence, i.e.\ $\lim_{q\downarrow 1}R_q(\rho\mmid\pi) = \KL(\rho\mmid\pi)$ and $(ii)$ as $q\to \infty$ it reduces to the $L^\infty$-norm, i.e.\ $\lim_{q\to \infty}R_q(\rho\mmid\pi) = \ln\|{\D\rho}/{\D\pi}\|_{L^\infty(\pi)}$.
It is also related to the $\chi^2$ divergence via $\chi^2(\rho \mmid \pi) + 1= \exp( R_2(\rho \mmid \pi) )$.

Providing convergence guarantees in R\'enyi divergence is of particular interest since it upper bounds many commonly used distance measures. Specifically, by Pinsker's inequality and the monotonicity of R\'enyi divergence, we have
$$2D_{\mathsf{TV}}(\rho, \pi)^2 \leq \KL(\rho\mmid\pi) \leq R_q(\rho\mmid\pi)\ \ \text{ for } \ \ q>1.$$
Notice that comparing the quadratic Wasserstein distance $\wass^2_2(\rho,\pi)$ with $R_q(\rho\mmid\pi)$ is more subtle. Under a PI with constant $\FIc_\PI$, or more broadly under finite fourth moments, one can write
$$\ln\left(1 + \tfrac{1}{2\FIc_\PI}\wass^2_2(\rho,\pi)\right) \leq R_2(\rho\mmid\pi)\ \ \text{ and }\ \ \ln\bigg(1 + \frac{\wass^4_2(\rho,\pi)}{4\Epi{\norm{x}^4}}\bigg) \leq R_2(\rho\mmid\pi),$$
respectively. The first inequality is due to \cite{liu2020poincare} and the second one
can be derived from the weighted total variation control on $\wass_2$ \citep[Proposition 7.10]{villani2003topics} along with the Cauchy-Schwartz inequality.
Hence, even under a WPI, a bound in R\'enyi divergence can be translated to a bound in $\wass_2$ (when $\Epi{\norm{x}^4} < \infty$), possibly at the expense of introducing additional dimension dependency into the bounds. It is worth highlighting that a distribution satisfying \eqref{eq:wpi} does not need to have any particular moment defined, in which case $\wass_2$ may be undefined, but its R\'enyi divergence of some order may still be well-defined.

\subsection{R\'enyi Convergence of the Langevin Diffusion}\label{sec:upperld}
Classically, convergence of the Langevin diffusion under \eqref{eq:wpi} is considered only in variance, or equivalently, the \(\chi^2\) divergence (see e.g.~\cite[Chapter 4]{wang2006functional} and~\cite[Chapter 7.5]{bakry2014analysis}). The following result characterizes its convergence in R\'enyi divergence which is stronger in the case of \(q > 2\).
\begin{theorem}\label{thm:diffusion_conv}
    Suppose $\pi$ satisfies \eqref{eq:wpi} for some $\FIc_\WPI$ and $\Phi(\cdot) = \osc(\cdot)^2$. For any $2 \leq q < q' \leq \infty$ such that $R_{q'}(\rho_0\mmid\pi) < \infty$, define $\delta_0 \coloneqq \exp(qR_{q'}(\rho_0\mmid\pi))$. Then, for any $r > 0$,
    \begin{equation*}
    R_q(\rho_t\mmid\pi) \leq \begin{cases}
    R_q(\rho_0\mmid\pi) - \frac{2-4r\delta_0}{\beta(r)q}t & \text{if} \quad R_q(\rho_0\mmid\pi),R_q(\rho_t\mmid\pi) \geq 1\\
    e^{-\tfrac{2t}{\beta(r)q}}(R_q(\rho_0\mmid\pi)-2r\delta_0) + 2r\delta_0 & \text{if} \quad R_q(\rho_0\mmid\pi) < 1,
    \end{cases}
\end{equation*}
where
\begin{equation}\label{eq:new-wpi-constant}
    \beta(r) \coloneqq \begin{cases}
    \FIc_\WPI(r) & \text{if} \quad q'=\infty\\ \FIc_\WPI\Big((r/5)^{\tfrac{q'}{q'-q}}\Big)\ln\Big((5/r)^{\tfrac{q'}{q'-q}} \lor 1\Big) & \text{if} \quad q' < \infty.
    \end{cases}
\end{equation}
Therefore, we have $R_q(\rho_T\mmid\pi) \leq \varepsilon$ whenever
\begin{equation}
    T \geq q\beta\Big(\frac{1}{4\delta_0}\Big)R_q(\rho_0\mmid\pi) + \frac{q}{2}\beta\Big(\frac{\varepsilon}{4\delta_0}\Big)\ln\Big(\frac{1}{\varepsilon}\Big).
\end{equation}
\end{theorem}
We emphasize that while the classical convergence results under \eqref{eq:wpi} \citep{wang2006functional, bakry2014analysis} require $R_\infty(\rho_0\mmid\pi) < \infty$ at initialization, our convergence guarantees hold as soon as the initial error satisfies $R_{q'}(\rho_0\mmid\pi)< \infty$ for some $q' > q$. Moreover, in the case where $\pi$ satisfies a PI, i.e.\ when $\FIc_\WPI$ is constant, we can remove the requirement of $R_{q'}(\rho_0\mmid\pi) < \infty$, and the above theorem recovers \cite[Theorem 3]{Vempala2019-jz} by defining $\beta(0) \coloneqq \lim_{r\to 0}\FIc_\WPI(r)$.

The proof of Theorem~\ref{thm:diffusion_conv} is presented in Appendix \ref{app:proof_diff_conv}, and it relies on a proof technique also used in  \cite{Vempala2019-jz,Chewi2021-vj}. However, because of the $\Phi(\cdot) = \osc(\cdot)^2$ term in \eqref{eq:wpi}, we additionally need to control oscillations of $\tfrac{\D\rho_t}{\D\pi}$ uniformly over the process. Via contraction properties for the semigroup, we can reduce such control to $\frac{\D\rho_0}{\D\pi} \in L^\infty(\pi)$ (or equivalently $R_\infty(\rho_0\mmid\pi) < \infty$). To further relax this assumption, one needs to obtain a WPI with a \emph{weaker} $\Phi$. Specifically, given an initial control of the type $\frac{\D\rho_0}{\D\pi} \in L^{q'}(\pi)$ for some $q' < \infty$, it suffices to obtain a WPI with $\Phi(\cdot) = \norm{\cdot}_{L^u(\pi)}^2$ (for mean-zero functions), where $u = \tfrac{2q'}{q}$. The following proposition, which is a consequence of a general $L_p$ decay result due to \cite{cattiaux2011central} and might be of independent interest, is crucial in the proof of Theorem~\ref{thm:diffusion_conv}. Specifically, it converts a WPI with $\Phi(\cdot) = \osc(\cdot)^2$ into a different WPI with $\Phi(\cdot) = \norm{\cdot}_{L^u(\pi)}^2$; thus, allows for a less restrictive initialization.
We defer the proof of this proposition to Appendix \ref{app:proof_prop_wpi_improved}.
\begin{proposition}\label{prop:wpi_improved}
    Suppose $\pi$ satisfies \eqref{eq:wpi} with $\Phi(\cdot) = \osc(f)^2$ and some $\FIc_\WPI(r)$.
    Then, for every $u > 2$, $\pi$ also satisfies \eqref{eq:wpi}
    with weighting \(\beta'\) and regularization function \(\Phi'\) such that
    \begin{equation}\label{eq:wpi_improved}
    \FIc'_\WPI(r) = \FIc_\WPI\left((r/5)^{\tfrac{u}{u-2}}\right)\ln\left((5/r)^{\tfrac{u}{u-2}} \lor 1\right)\ \ \text{ and }\ \ \Phi'(\cdot) = \norm{f - \Epi{f}}_{L^u(\pi)}^2.
    \end{equation}
    Additionally, provided that $\pi$ does not satisfy a PI, $\pi$ cannot satisfy a WPI with $\Phi=\Phi'$ for $u=2$.
\end{proposition}
The last statement of the above proposition clarifies why we need to choose $q' > q$. In order to establish guarantees with $q' = q$, one would be required to obtain a WPI with $\Phi(\cdot) = \norm{\cdot}_{L^2(\pi)}^2$, which is equivalent to a PI.

\section{Langevin Monte Carlo for Heavy-Tailed Targets}\label{sec:conv}
In this section, we present our main convergence guarantees for LMC when the target satisfies \eqref{eq:wpi} and $\nabla V$ is $s$-H\"older continuous for some $s\in(0,1]$, that is,
\begin{equation}\label{eq:holder}
\norm{\nabla V(x) - \nabla V(y)} \leq L\norm{x-y}^s \quad \forall x,y \in \reals^d.
\tag{$s$-H\"older}
\end{equation}
The case $s=1$ corresponds to the ubiquitous Lipschitz continuity where the potential is smooth, and the regime where $s<1$ is often termed as weak smoothness. Since our main focus is potentials that do not satisfy \eqref{eq:pi}, any order H\"older continuity is feasible.

Below, we state our main convergence result for a generic \eqref{eq:wpi} with an unspecified $\FIc_\WPI$.
We will explicitly derive its implications for specific targets in the subsequent sections.

\begin{theorem}\label{thm:lmc_upper}
Suppose $\pi\propto e^{-V}$ satisfies \eqref{eq:wpi} for some $\FIc_\WPI$ and $\Phi(\cdot) = \osc(\cdot)^2$,
and $\nabla V$ is \eqref{eq:holder} continuous, with $\nabla V(0) = 0$ for simplicity. For any $q \in [2,\infty), q' \in (2q-1,\infty]$ such that $R_{q'}(\mu_0\mmid\pi) < \infty$, define $\delta_0 \coloneqq \exp((2q-1)R_{q'}(\mu_0\mmid\pi))$, $m \coloneqq \tfrac{1}{2}\inf\{R \, : \, \pi(\norm{x} \geq R) \leq \tfrac{1}{2}\}$, and
$$T \coloneqq (2q-1)\left\{\beta\left(\frac{1}{4\delta_0}\right)R_{2q-1}(\mu_0\mmid\pi) + \beta\left(\frac{\varepsilon}{8\delta_0}\right)\ln
\left(\frac{2}{\varepsilon}\right)\right\},$$
for $\varepsilon \leq q^{-1}$, where $\beta$ is as in \eqref{eq:new-wpi-constant}.
Let $\hat\pi$ denote a modified version of $\pi$ (explicitly defined in \eqref{eq:modified-target}) and assume, for simplicity, that $\varepsilon^{-1},m,L,T,R_2(\mu_0\mmid\hat{\pi})\geq 1$.
Then, for a sufficiently small step size $h$,
denoting by $ \mu_{N}$, the law of the $N$-th iterate of LMC initialized at $\rho_0$,
after
\begin{equation*}
    \! N \!=\! \Theta_s\!\left(\!\frac{T^{1+1/s}dq^{1/s}L^{2/s}}{\varepsilon^{1/s}}\max\!\left\{\!1,\frac{\varepsilon^{1/(2s)}m^s}{L^{1/s-1}T^{1/(2s)}d},\frac{\varepsilon^{1/(2s)}R_2(\mu_0\mmid\hat{\pi})^{s/2}}{L^{1/s-1}T^{(1-s^2)/(2s)}d}\ln\left(\frac{qTLR_2(\mu_0\mmid\hat{\pi})}{\varepsilon}\right)^{\!\!s/2}\!\right\}\!\right)
\end{equation*}
iterations of LMC, we obtain $R_q(\mu_N\mmid\pi) \leq \varepsilon$.
\end{theorem}

We make a few remarks.
First, it is possible to find an initialization $\mu_0$, e.g.\ isotropic Gaussian, such that $R_q(\mu_0\mmid\pi),R_q(\mu_0\mmid\hat{\pi}) \leq \tilde{\mathcal{O}}(d)$, with details provided in Lemma \ref{lem:init}. In such a scenario, up to log factors, the last term in the maximum will never dominate. Additionally, the middle term will never dominate for sufficiently small $\varepsilon$, and in fact, in our model examples, it will not dominate even for $\varepsilon = 1$ due to proper control on $m$. Then our convergence rate reads
$$N = \tilde{\Theta}_s\Bigg(\frac{q^{1+2/s}L^{2/s}d\big\{d\beta(\tfrac{1}{4\delta_0}) + \beta(\tfrac{\varepsilon}{8\delta_0})\big\}^{1+1/s}}{\varepsilon^{1/s}}\Bigg).$$
In the case where $\pi$ satisfies \eqref{eq:pi}, we set $\beta(r) = \FIc_{\PI}$ for any $r > 0$, and the above rate reduces to
the rate implied by \cite[Theorem 7]{Chewi2021-vj}.

The proof of Theorem~\ref{thm:lmc_upper} is given in Appendix~\ref{app:proof_lmc_upper}, and it is based on Theorem~\ref{thm:diffusion_conv} and the Girsanov argument used in \cite{Chewi2021-vj}. Two key distinctions are $(\mathrm{i})$ our analysis is tailored for heavy-tailed targets, hence does not require any moment order to be defined,
and also requires a finer control of different error terms to mitigate suboptimal rates, and $(\mathrm{ii})$ it exploits the continuous time convergence under WPI, rather than stronger functional inequalities such as PI.

\subsection{Examples}\label{sec:ex-upper}
In this section, we focus on various heavy-tailed targets that do not satisfy \eqref{eq:pi}.
In particular, we consider sampling from Cauchy-type measures in Section~\ref{sec:upper_examples_cauchy}, which may not even have a defined expectation or any order moment for that matter; yet, we are able to provide convergence guarantees for LMC in R\'enyi divergence of any finite order.

\subsubsection{Potentials with sub-linear tails}\label{sec:upper_examples_sub_linear}
Consider the measure $\pi_\alpha \propto \exp(-V_\alpha)$ where $V_\alpha(x) = (1 + \norm{x}^2)^{\alpha/2}$ with $\alpha \in (0,1)$. This potential satisfies \eqref{eq:holder} with $s=1$ and $L=1$. We analyze this potential as a substitute for $\norm{x}^\alpha$ since the latter does not have continuous gradients, while the former still behaves similar to $\norm{x}^\alpha$ for large $\norm{x}$. In the following Proposition, we present the $\beta_\WPI$ estimate for this potential.
\begin{proposition}\label{prop:WPI subexp true weights} The measure $\pi_\alpha(x)\propto \exp(-V_\alpha)$ with $\alpha\in (0,1)$ satisfies \eqref{eq:wpi} with
\begin{align}\label{eq:WPI constant subexp}
    \FIc_{\WPI}(r)=\inf_{\gamma \in (0,2\alpha]}C_\alpha \left( d^{\frac{2(2-2\alpha+\gamma)}{\gamma}}+\ln\left(r^{-1}\right)^{\frac{2-2\alpha+\gamma}{\alpha}}\right)\ \ \text{ and }\ \ \Phi(\cdot) = \osc(\cdot)^2,
\end{align}
where $C_\alpha$ is a constant depending only on $\alpha$.
\end{proposition}
It is worth noting that this estimate is polynomial in dimension, improving the implicit and potentially exponential dependency implied by \cite[Proposition 5.6]{cattiaux2010functional} at the expense of introducing an additional factor of $\ln(1/r)^2$ into the estimate. Specifically, they show that any potential of the form $V(x) = \psi(x)^\alpha$ for some convex non-negative $\psi$ and $\alpha \in (0,1)$ satisfies a WPI with $\FIc_\WPI(r) = C_{\WPI}(1 + \ln(r^{-1})^{2(1/\alpha-1)})$ where $C_{\WPI} = C_\WPI(d,\alpha)$ is implicit in dimension, and in general it is not known how to achieve a dependency better than exponential via such techniques.
Invoking Theorems~\ref{thm:diffusion_conv} and \ref{thm:lmc_upper} with the estimate for $\FIc_\WPI$ given by Proposition~\ref{prop:WPI subexp true weights}, we can establish the following convergence guarantees for LMC and the Langevin diffusion.
\begin{corollary}
Consider the setting of Theorem \ref{thm:lmc_upper} with the target measure $\pi_\alpha \propto \exp(-V_\alpha)$. Denoting the distribution of the Langevin diffusion at time $T$ with $\rho_T$, we have $R_q(\rho_T\mmid\pi_\alpha) \leq \varepsilon$ whenever\
$$T \geq C_{\alpha,q}\!\inf_{\gamma \in (0,2\alpha]}\!\!\left\{\left(d^{\tfrac{2}{\gamma}(2-2\alpha+\gamma)}\!\! + R_\infty(\rho_0\mmid\pi)^{\tfrac{2-2\alpha+\gamma}{\alpha}}\right)\left(R_q(\rho_0\mmid\pi) + \ln(1/\varepsilon)\right) + \ln(1/\varepsilon)^{\tfrac{2-\alpha+\gamma}{\alpha}}\!\right\}\!.$$
Further, denoting the distribution of LMC after $N$ iterations with $\mu_N$, we have $R_q(\mu_N\mmid\pi_\alpha) \leq \varepsilon$ if
$$N = \tilde{\Theta}_{\alpha,q}\left(\frac{(d^{4/\alpha} + R_\infty(\mu_0\mmid\pi_\alpha)^{4/\alpha})R_{2q-1}(\mu_0\mmid\pi_\alpha)^2d}{\varepsilon}\max\left\{1,\frac{\sqrt{\varepsilon R_2(\mu_0\mmid\hat{\pi}_\alpha)}}{d}\right\}\right).$$
\end{corollary}
In particular, with $\rho_0 = \mu_0 = \mathcal{N}(0,I_d)$, and $\gamma = 2\alpha$, and using the bound on the initial R\'enyi divergence provided by Corollary \ref{cor:gaussian_sublin_renyi} and Lemma \ref{lem:initial_renyi_alt}, we obtain
$$T \gtrsim C_{\alpha,q}\left(d^{2/\alpha + 1} + d^{2/\alpha}\ln(1/\varepsilon) + \ln(1/\varepsilon)^{2/\alpha + 1}\right),\ \ \text{ and }\ \ N = \tilde{\Theta}_{\alpha,q}\big(\tfrac{d^{4/\alpha + 3}}{\varepsilon}\big),$$
where $\gtrsim$ hides polylog factors in $d$.

As a final remark, if instead of the $\FIc_\WPI$ estimate of Proposition \ref{prop:WPI subexp true weights}, we use the estimate from \cite{cattiaux2010functional} together with Theorem~\ref{thm:lmc_upper}, we can obtain a rate for LMC which reads $N = \tilde{\Theta}_{\alpha,q}\big(C_\WPI^2\tfrac{d^{4/\alpha - 1}}{\varepsilon}\big)$. In particular, this can be seen as a smooth extrapolation of the rate given by \cite{Chewi2021-vj} for the case of \(\alpha \in [1, 2]\). By showing that an LOI holds with some constant $C_\LOI$, they obtain rates identical to ours once $C_\WPI$ is replaced by $C_\LOI$, and the two rates match at $\alpha = 1$ as $C_\LOI$ and $C_\WPI$ will be equivalent up to an absolute constant in this case. However, when $\alpha < 1$, $C_\WPI$ of \cite{cattiaux2010functional} has an implicit and potentially exponential dependence on dimension. Although our estimate in Proposition~\ref{prop:WPI subexp true weights} improves this to polynomial, due to the additional log factor introduced into the bound, the rate will no longer smoothly extrapolate the rate of \cite{Chewi2021-vj} to the regime $\alpha < 1$.

\subsubsection{Potentials with logarithmic tails: Cauchy-type measures} \label{sec:upper_examples_cauchy}
In this section, we consider Cauchy-type measures of the form $\pi_\nu \propto \exp(-V_\nu)$ where $V_\nu(x) = \tfrac{d+\nu}{2}\ln(1 + \norm{x}^2)$ with $\nu > 0$, which is H\"older continuous with $s=1$ and $L \leq d+\nu$. In particular, $\pi_\nu$ belongs to the family of $d$-dimensional Student-t distributions with $\nu$ degrees of freedom; see e.g.\ \cite{Jarner2007-du, Kamatani2014-ri,He2022-np,yang2022stereographic} for sampling from such distributions. Cauchy-type measures only have finite moments of order less than $\nu$.

The following result presents a sharp estimate of $\FIc_\WPI$ for $\pi_\nu$.
\begin{proposition}\label{prop:wpi_t}
    The measure $\pi_\nu \propto \exp(-V_\nu)$ for $\nu>0$ satisfies \eqref{eq:wpi} with
    \begin{equation}
        \beta_\WPI(r) = \frac{2}{\nu} + 2\left(\frac{d}{\nu} + 1\right)r^{-2/\nu} \ \ \text{ and }\ \ \Phi(\cdot) = \osc(\cdot)^2.
    \end{equation}
\end{proposition}
Similar to the previous case, the above estimate improves the potentially exponential dimension dependence in \cite[Proposition 5.4]{cattiaux2010functional} to linear, while keeping $r$ dependency the same.
Employing the above estimate in Theorems~\ref{thm:diffusion_conv} and
\ref{thm:lmc_upper}, we obtain the following rate for LMC and the Langevin diffusion.
\begin{corollary}\label{cor:cauchy_upper}
    Consider the setting of Theorem \ref{thm:lmc_upper} with the target measure $\pi_\nu \propto \exp(-V_\nu)$. Denoting the distribution of the Langevin diffusion at time $T$ with $\rho_T$, we have $R_q(\rho_T\mmid\pi_\nu) \leq \varepsilon$ whenever
    $$T \geq C_\nu qd\exp\left(\frac{2qR_\infty(\rho_0\mmid\pi)}{\nu}\right)\left(R_q(\rho_0\mmid\pi) + (1/\varepsilon)^{2/\nu}\ln(1/\varepsilon)\right).$$
Furthermore, let $\nu \leq cd$ for some absolute constant $c$, and denote the distribution of LMC after $N$ iterations with $\mu_N$. Then, we have $R_q(\mu_N\mmid\pi_\nu) \leq \varepsilon$ whenever
    $$\!N\! =\! \tilde{\Theta}_{\nu,q}\!\left(\!\frac{d^5e^{\tfrac{4(2q-1)}{\nu}R_\infty(\mu_0\mmid\pi)}}{\varepsilon}\left(R^2_{2q-1}(\mu_0\!\mmid\pi) \!+\! \Big(\frac1{\varepsilon}\Big)^{4/\nu}\right)\max\left\{\!1,\frac{\sqrt{\varepsilon R_2(\mu_0\!\mmid\hat{\pi})R_\infty(\mu_0\!\mmid\pi)}}{d}\right\}\!\!\right)\!\!.$$
\end{corollary}
We highlight that the dependence on the initial error is exponential in contrast to the polynomial dependence in the previous case. In other words, with a naive initialization, in the worst case, this dependency will negatively impact the complexity of sampling from Cauchy-type measures.
However, if one initializes with an isotropic Gaussian with an appropriately scaled variance,
using the estimates on the initial R\'enyi divergence in Corollary \ref{cor:gaussian_cauchy_renyi} and Lemma \ref{lem:initial_renyi_alt}, we can obtain
$$T \geq C_{\nu,q}d^{2q/\nu + 1}\left(\ln d + (1/\varepsilon)^{2/\nu}\ln(1/\varepsilon)\right) \ \ \text{ and }\ \
N = \tilde{\Theta}_{\nu,q}\big(d^{\frac{4(2q-1)}{\nu} + 5}/\varepsilon^{4/\nu + 1}\big).
$$
Finally, we remark that such an initialization may not be available in general; thus, the exponential dependence might be unavoidable.
\section{Lower Bounds for LMC via Variance Decay}\label{sec:lower_bounds}
As demonstrated in the examples, the convergence guarantees given in Section \ref{sec:conv} have worse dependence on the initial divergence when the tails of the target are heavier, with a sharp transition at (Cauchy-type) logarithmic tails, in which case the dependence on the initial error becomes exponential. In this section, we show that this is not due to a limitation in the analysis, but is in fact a property of LMC in heavy-tailed settings. We present a method for developing lower bounds for the convergence rate of LMC, with the goal of observing a similar dependence on the initial divergence and in particular, establishing that the exponential dependency for Cauchy-type measures is unavoidable in the worst case.
First, we introduce the notation of complexity that we lower bound.
\begin{definition}
Let $\mathcal{D}$ and $\mathcal{D}'$ denote two divergences.
\begin{itemize}[itemsep=-.1cm,leftmargin=.2in]
\item The iteration complexity of LMC, denoted by $\mathscr{C}^{\text{LMC}}_{\mathcal{D}, \mathcal{D}'}(\pi,h,\Delta_0,\varepsilon)$, is the smallest $k \in \N$ such for any $\mu_0$ satisfying $\mathcal{D}'(\mu_0 \mmid \pi) \leq \Delta_0$, LMC with step-size \(h\) satisfies
$\mathcal{D}(\mu_k\mmid\pi) \leq \varepsilon$.

\item Similarly, the time complexity of the Langevin diffusion, $\mathscr{C}^{\text{LD}}_{\mathcal{D}, \mathcal{D}'}(\pi,\Delta_0,\varepsilon)$, is the smallest $T \in \R_+$ such that for any $\rho_0$ satisfying \(\mathcal{D}'(\rho_0 \| \pi) \leq \Delta_0\),
the diffusion satisfies
$
    \mathcal{D}(\rho_T \| \pi) \leq \varepsilon.
$
\end{itemize}

\end{definition}

For example, \(\mathcal{D}\) and \(\mathcal{D}'\) could be the KL divergence, the order-$q$ R\'enyi divergence or the $p$-Wasserstein distance. By having this notion of complexity uniformly over all initial distributions whose initial error does not exceed $\Delta_0$, we mirror the types of bounds given in Section \ref{sec:conv} and in the literature on LMC more broadly (see e.g.~\cite[Definition 5]{Chewi2022-fv}). Indeed, in Section~\ref{sec:ex-upper}, we provided upper bounds for this quantity with \(\mathcal{D} = R_q\) and \(\mathcal{D'} = R_{q'}\).

Our methodology for developing lower bounds is based on the observation that in order for LMC to be close to the target in some divergence, its moments should match those of the target. We study the second moments of LMC and the Langevin diffusion and find that, by tracking their evolution through differential inequalities, we can relate their convergence to the convergence of processes that are more tractable. %
To relate this to the R\'enyi divergence, we rely on the strategy of variational representations; the most famous of these, the Donsker-Varadhan theorem, has the KL divergence represented as a supremum over a functional. We use a similar representation for the \(q\)-R\'enyi divergence with \(q > 1\), which is due to \cite{Birrell2020-gu}, with a particular choice of test function to obtain
\begin{equation*}
    R_q(\rho \| \pi) \geq \ln \Big ( {\rho(\|\cdot\|^2)^{\frac{q}{q-1}}}\big /{\pi(\|\cdot\|^{\frac{2q}{q - 1}})} \Big )\ \ \text{ whenever }\ \ \pi(\|\cdot\|^{2q/(q-1)}) < \infty.
\end{equation*}

Therefore, a strategy for obtaining lower bounds in the case where \(\E\|x_0\|^2\) is large follows: $(i)$ relate the decay of the R\'enyi divergence to the decay of \(\E\|x_k\|^2\), $(ii)$ lower bound the number of iterations required for \(\E\|x_k\|^2\) to decay by the number of iterations for a more tractable process (e.g.\ gradient descent) to converge. To strengthen the lower bounds obtained, we use a similar approach to develop an upper bound on the step size necessary to obtain a given accuracy. %
Details for the methodology can be found in Appendix~\ref{app:lower_methodology}.

\subsection{Heavy-Tailed Potentials and Slow Starts}
In order to obtain lower bounds, we will make an assumption of the type
$$\norm{\nabla V(x)} \lesssim \norm{x}^{\alpha - 1},$$
with $\alpha \in [0,2]$, for sufficiently large $\norm{x}$, which will capture potentials that grow no faster than the order of $\norm{x}^\alpha$ when $\alpha > 0$, and $\ln(\norm{x})$ when $\alpha = 0$.
Under such a condition, we show that as $\alpha \to 0$, the dependence on initial error deteriorates. In particular, the dependence is logarithmic with $\alpha = 2$, becomes polynomial with $\alpha \in (0,2)$, and finally turns exponential with $\alpha = 0$, as outlined below.
\begin{theorem}[Three-step Phase Transition]\label{thm:lower_bound}
Let \(q \in (1, \infty)\), \(q' \in [1, \infty]\) and the moment condition \(\pi(\|\cdot\|^{2q/(q-1)}) < \infty\) hold. Suppose that
\begin{equation}\label{eq:grad_v_growth}
    \norm{\nabla V(x)} \leq \frac{b\norm{x}}{(1 + \norm{x}^2)^{1-\alpha/2}},
\end{equation}
with $\alpha \in [0, 2]$ and $b > 0$, and let $\nu \coloneqq b - d$. Then for all sufficiently large $\Delta_0$ (see \eqref{eq:delta0_lower_cauchy}, \eqref{eq:delta0_lower_sublin}, and \eqref{eq:delta0_lower_gaussian} for respective cases), the time complexity of the Langevin diffusion, and the iteration complexity of LMC, for obtaining an accuracy of \(1\) in \(q\)-R\'enyi divergence, is lower bounded as follows
\begin{alignat*}{5}
    &\alpha = 2: & \quad & \mathscr{C}^{\text{LMC}}_{R_q, \KL}(\pi, h, \Delta_0, 1) \gtrsim \frac{\ln(\Delta_0)}{h}, & \quad & \mathscr{C}^{\text{LD}}_{R_q, \KL}(\pi, \Delta_0, 1) \gtrsim \ln(\Delta_0),\\[.1em]
    &\alpha \in(0,2): & \quad & \ccplx{q}{q'}(\pi, h, \Delta_0, 1) \gtrsim \frac{d^{1-\alpha/2}\Delta_0^{\frac{(2 - \alpha)^2}{2\alpha}}}{h}, & \quad & \tcplx{q}{q'}(\pi, \Delta_0, 1) \gtrsim d^{1-\alpha/2}\Delta_0^{\frac{(2 - \alpha)^2}{2\alpha}},\\[.8em]
    &\alpha = 0: & \quad & \ccplx{q}{q'}(\pi, h, \Delta_0, 1) \gtrsim \frac{d}{\nu h}\exp(\Delta_0/\nu), & \quad & \tcplx{q}{q'}(\pi, \Delta_0, 1) \gtrsim \frac{d}{\nu}\exp(\Delta_0/\nu),
\end{alignat*}
where $\gtrsim$ hides a constant depending only on $b$ and $\alpha$. The lower bounds hold for all $h > 0$, except when $\alpha = 2$ where they hold for $h < b^{-1}$.

\end{theorem}
In Theorem \ref{thm:lower_bound}, the case of \(\alpha = 2\) corresponds to tails no lighter than Gaussian and indeed, the lower bounds that we obtain reproduce the dependence of \(\Delta_0\) known in the Gaussian setting (\cite{Vempala2019-jz}). The case of \(\alpha \in (0,2)\) corresponds to potentials with tail growth similar to \(\|x\|^\alpha\) and the case of \(\alpha = 0\) corresponds to Cauchy-type logarithmic tails. Indeed, the generalized Cauchy potentials $V_\nu$ and the sub-linear potentials $V_\alpha$ from Section \ref{sec:ex-upper} satisfy \eqref{eq:grad_v_growth} with $b = d + \nu$ and $b = \alpha$ respectively. For these examples, our lower bounds recover the polynomial and exponential dependence on the initial error given in the upper bounds of Section~\ref{sec:conv}. More generally, we point out that \eqref{eq:grad_v_growth} is satisfied for any smooth $\nabla V$ with $\nabla V(0) = 0$ and $\norm{\nabla V(x)} \lesssim \norm{x}^{\alpha-1}$ for large $\norm{x}$. Similar to the upper bounds, the implicit assumption \(\nabla V(0) = 0\) is made only for the simplicity of presentation.
In fact, to achieve the most generality, even this assumption can be relaxed to a bound of the type $\binner{\nabla V(x)}{x} \lesssim \norm{x}^{\alpha}$ for large $\norm{x}$ while recovering similar results. One can interpret such a condition as \emph{reversed dissipativity}. While dissipativity is used in the literature to ensure a lower bound on tail growth for obtaining upper bounds on moments (see e.g.\ \cite{raginsky2017non,erdogdu2018global,erdogdu2022convergence,Farghly2021-xg}), the reversed condition imposes an upper bound on the tail growth that leads to obtaining lower bounds on the moments.

The above result highlights one reason why LMC can be seen to perform worse in heavy-tailed settings. As the tail becomes heavier, LMC exhibits a slow start behavior by having a worse dependence on the initial divergence. Showing that this is a property of the Langevin diffusion and not due to discretization
motivates using alternative diffusions for sampling from heavy-tailed targets (see, for example,~\cite{li2019stochastic,simsekli2020fractional,He2022-np,zhang2022ergodicity}).

To complete our exposition, %
we note that while the results of Theorem \ref{thm:lower_bound} are stated for any fixed choice of step size, in practice we have to ensure step size is small enough for the discretization to not harm the convergence of LMC. This usually leads to additional dependence on dimension or final accuracy. In a special case where the target potential is radially symmetric, the following proposition provides a general tool for determining suitable ranges of step size for LMC, hence completing the complexity lower bound in conjunction with Theorem \ref{thm:lower_bound}.
\begin{proposition}[Step-size upper bound]\label{prop:step_size_bound_summary}
Suppose the potential is radially symmetric with $V(x) = f(\norm{x}^2)$ and \(g: \R_+ \to \R_+\) given by \(g(r) = (1 - 2 h f'(r))^2 r\) is convex and non-decreasing. Let \(\varepsilon > 0, q > 1\). Suppose that \(\Earg{\|x_0\|^2} > \sigma^2_\varepsilon\) where
$
    \sigma^2_\varepsilon = e^{\frac{q-1}{q} \varepsilon} \pi \big ( \|\cdot\|^{\frac{2q}{q-1}} \big )^{\frac{q-1}{q}}.
$
Then it must hold that,
\begin{equation*}
    \inf_{k \in \N} R_q(\rho_k \| \pi) < \varepsilon \implies h \leq \frac{1}{f'(\sigma^2_\varepsilon)} \bigg ( 1 - \frac{d}{2 f'(\sigma^2_\varepsilon) \sigma^2_\varepsilon} \bigg ).
\end{equation*}
\end{proposition}

We conclude our discussion by considering an example that demonstrates an application of the main tools developed in this paper.
Specifically, we consider the setting of Section \ref{sec:upper_examples_cauchy} and recall the Cauchy-type target, \(\pi_\nu\). The following corollary, which gives a sharp characterization of the convergence behavior, is a direct consequence of the lower bound from Theorem \ref{thm:lower_bound} and the upper bound from Corollary \ref{cor:cauchy_upper}.
\begin{corollary}
Let $q \geq 2$ and suppose that \(\nu > \frac{2q}{q-1}\).
Then for any \(\Delta_0 \geq C_q(1 + \nu \ln(d + \nu))\), we have the following bound for the Langevin diffusion,
\begin{equation*}
    \frac{d}{4\nu}\exp\left(\frac{\Delta_0}{\nu}\right) \leq \tcplx{q}{\infty}(\pi, \Delta_0, 1) \leq \frac{d}{\nu}\exp\left(\frac{C_q\Delta_0}{\nu}\right).
\end{equation*}
Further, LMC with an appropriate step size $h > 0$ satisfies
\begin{equation*}
    \frac{d}{4h\nu}\exp\left(\frac{\Delta_0}{\nu}\right) \leq \ccplx{q}{\infty}(\pi, h, \Delta_0, 1)
     \leq \frac{d}{h\nu}\exp\left(\frac{C_q\Delta_0}{\nu}\right).
\end{equation*}
\end{corollary}

The above result states that the exponential dependence on the initial error for LMC and the diffusion is unavoidable, unless there is a good initialization available.
Notice that, by
Proposition~\ref{prop:step_size_bound_summary}, to obtain $R_q(\mu_N\mmid\pi) \leq \varepsilon$, step size needs to be sufficiently small.
Incorporating this into the above bound, we obtain the following lower bound for the iteration complexity of LMC when \(d \geq \nu - p\):
\begin{equation*}
    \ccplx{q}{q'}(\pi, h, \Delta_0, \varepsilon) \geq \frac{(d+\nu)^2}{48e\nu} \min \bigg \{ \varepsilon^{-1}, \frac{\nu-p}{p}, \frac{(\nu - p)d}{(2+p) \nu} \bigg \} \exp\left(\frac{\Delta_0}{\nu}\right),
\end{equation*}
where $p = \frac{2q}{q-1}$. Note that the step-size bound leads to additional dependence on \(d\) and \(\varepsilon\). However, the dependence on \(\varepsilon\) only reveals itself when \(\nu\) and \(d\) are relatively large.

\section{Conclusion}
We provided convergence guarantees for LMC and Langevin diffusion, for target distributions $\pi\propto e^{-V}$ satisfying a WPI.
This corresponds to potentials with tails that behave like $V(x) \sim \|x\|^\alpha$ for $\alpha \in(0,1]$, and also covers Cauchy-type densities in the case $\alpha\downarrow 0$. Our results push the program initiated by \cite{Vempala2019-jz} to its limits; by providing explicit WPI constants for specific examples, we obtained guarantees demonstrating that targets with heavier tails lead to a worse
dependence on the initial error. Particularly, the dependence on initial error is a polynomial of order $\tfrac{(2-\alpha)^2}{2\alpha}$ for $\alpha > 0$, with a phase transition at $\alpha = 0$, i.e.\ Cauchy-type logarithmic tails, for which the dependence becomes exponential. Additionally, we established lower bounds under generic tail growth conditions that asserted such dependence on the initial error is unavoidable unless suitable warm start initializations are available.

Our quantitative results (via upper and lower bounds) highlight the precise limitations of LMC for heavy-tailed sampling and provide further motivations to develop a complete complexity theory of heavy-tailed sampling by discretizing other diffusions like specific It\^o diffusions or stable-driven diffusions, an area of research which is still in its infancy.

One limitation of our upper bounds is the fact that the step size needs to be chosen in advance according to the number of iterations, and with a fixed step size, the upper bounds diverge as $N \to \infty$. However, as pointed out by \cite{Chewi2021-vj}, this is a general limitation of any analysis that does not assume the log-Sobolev inequality. We leave the stability of fixed step size LMC in the number of iterations under heavy-tailed targets as an open direction for future research.

\subsection*{Acknowledgments}
TF was supported by the Engineering and Physical Sciences Research Council (EP/T5178) and by the DeepMind scholarship. KB was supported in part by NSF grant DMS-2053918. MAE was supported by NSERC Grant [2019-06167], CIFAR AI Chairs program, CIFAR AI Catalyst grant, and Data Sciences Institute at University of Toronto.

\bibliographystyle{amsalpha}
\bibliography{references}

\newpage

\appendix
\section{Proofs of Sections~\ref{sec:diff} and~\ref{sec:conv}}
\input{appendices/proofs_upper.tex}

\section{Proofs of Section~\ref{sec:lower_bounds}}
\input{appendices/proofs_lower_bounds.tex}

\section{Auxiliary Lemmas}
\input{appendices/auxiliary.tex}

\end{document}

%% file: _macros.tex
\usepackage{comment,url,algorithm,algorithmic,graphicx,subcaption,relsize}
\usepackage{amssymb,amsfonts,amsmath,amsthm,amscd,dsfont,mathrsfs,mathtools,microtype,nicefrac,pifont}
\usepackage{float,psfrag,epsfig,color,url,hyperref}
\usepackage{upgreek}
\usepackage[dvipsnames]{xcolor}
\usepackage{epstopdf,bbm,mathtools}
\usepackage[toc,page]{appendix}
\usepackage[shortlabels]{enumitem}

\usepackage[top=1in, bottom=1in, left=1in, right=1in]{geometry}

\def\balign#1\ealign{\begin{align}#1\end{align}}
\def\baligns#1\ealigns{\begin{align*}#1\end{align*}}
\def\balignat#1\ealign{\begin{alignat}#1\end{alignat}}
\def\balignats#1\ealigns{\begin{alignat*}#1\end{alignat*}}
\def\bitemize#1\eitemize{\begin{itemize}#1\end{itemize}}
\def\benumerate#1\eenumerate{\begin{enumerate}#1\end{enumerate}}

\newenvironment{talign*}
 {\csname align*\endcsname}
 {\endalign}
\newenvironment{talign}
 {\csname align\endcsname}
 {\endalign}

\def\balignst#1\ealignst{\begin{talign*}#1\end{talign*}}
\def\balignt#1\ealignt{\begin{talign}#1\end{talign}}

\let\originalleft\left
\let\originalright\right
\renewcommand{\left}{\mathopen{}\mathclose\bgroup\originalleft}
\renewcommand{\right}{\aftergroup\egroup\originalright}

\def\tinycitep*#1{{\tiny\citep*{#1}}}
\def\tinycitealt*#1{{\tiny\citealt*{#1}}}
\def\tinycite*#1{{\tiny\cite*{#1}}}
\def\smallcitep*#1{{\scriptsize\citep*{#1}}}
\def\smallcitealt*#1{{\scriptsize\citealt*{#1}}}
\def\smallcite*#1{{\scriptsize\cite*{#1}}}

\def\reals{\mathbb{R}} %

\def\<{\left\langle} %
\def\>{\right\rangle}

\newcommand{\binner}[2]{\left\langle{#1},{#2}\right\rangle} %

\def\Earg#1{\E\left[{#1}\right]}

\DeclareSymbolFont{rsfs}{U}{rsfs}{m}{n}
\DeclareSymbolFontAlphabet{\mathscrsfs}{rsfs}

\newcommand{\pderiv}[2]{\frac{\partial #1}{\partial #2}} %

\ifdefined\nonewproofenvironments\else
\ifdefined\ispres\else
\newtheorem{theorem}{Theorem}
\newtheorem{lemma}[theorem]{Lemma}
\newtheorem{corollary}[theorem]{Corollary}
\newtheorem{definition}[theorem]{Definition}

\renewenvironment{proof}{\noindent\textbf{Proof.}\hspace*{.3em}}{\qed \vspace{.1in}}
\newenvironment{proof-sketch}{\noindent\textbf{Proof Sketch}
  \hspace*{1em}}{\qed\bigskip\\}
\newenvironment{proof-idea}{\noindent\textbf{Proof Idea}
  \hspace*{1em}}{\qed\bigskip\\}
\newenvironment{proof-of-lemma}[1][{}]{\noindent\textbf{Proof of Lemma {#1}}
  \hspace*{1em}}{\qed\\}
\newenvironment{proof-of-theorem}[1][{}]{\noindent\textbf{Proof of Theorem {#1}}
  \hspace*{1em}}{\qed\\}
\newenvironment{proof-of-proposition}[1][{}]{\noindent\textbf{Proof of Proposition {#1}}
  \hspace*{1em}}{\qed\\}
\newenvironment{proof-attempt}{\noindent\textbf{Proof Attempt}
  \hspace*{1em}}{\qed\bigskip\\}

\newenvironment{remark}{\noindent\textbf{Remark.}
  \hspace*{0em}}{\smallskip}%

\fi

\newtheorem{proposition}[theorem]{Proposition}

\newtheorem*{assumption*}{Assumptions}

\theoremstyle{definition}

\fi
\makeatletter
\@addtoreset{equation}{section}
\makeatother

\hypersetup{
  colorlinks,
  linkcolor={red!50!black},
  citecolor={blue!50!black},
  urlcolor={blue!80!black}
}

%% file: appendices/proofs_upper.tex
\subsection{Proof of Theorem \ref{thm:diffusion_conv}}\label{app:proof_diff_conv}
Following \cite{Vempala2019-jz}, define the following quantities
\begin{equation*}
    F_q(\rho\mmid\pi) \coloneqq \E_{\pi} \left(\frac{\rho}{\pi}\right)^q, \qquad G_q(\rho\mmid\pi) \coloneqq \E_{\pi} \Big [ \left(\frac{\rho}{\pi}\right)^q\norm{\nabla\log\frac{\rho}{\pi}}^2 \Big ] = \frac{4}{q^2}\E_{\pi} \bigg [ \norm{\nabla\left(\frac{\rho}{\pi}\right)^{q/2}}^2 \bigg ].
\end{equation*}
Then \cite[Lemma 6]{Vempala2019-jz} shows that along the Langevin diffusion
\begin{equation}
    \pderiv{R_q(\rho_t\mmid\pi)}{t} = -\frac{qG_q(\rho_t\mmid\pi)}{F_q(\rho_t\mmid\pi)}.\label{eq:renyi_time_derivative}
\end{equation}
Furthermore, we have the following estimates on the quantities appearing in our functional inequalities.
\begin{lemma}[\cite{Vempala2019-jz}]
\label{lem:quantity_estimates}
Let $f=\left(\tfrac{\rho_t}{\pi}\right)^{q/2}$. Then, for $q \geq 2$,
\begin{align*}
    \Epi{\norm{\nabla f}^2} &= \frac{q^2}{4}G_q(\rho_t\mmid\pi),\\
    \var_\pi(f) &\geq F_q(\rho_t\mmid\pi)(1-e^{-R_q(\rho_t\mmid\pi)}).
\end{align*}
\end{lemma}
\begin{proof}
The equality follows by definition.
For the inequality,
\begin{align*}
    \var_\pi(f) &= F_q(\rho_t\mmid\pi) - F_{q/2}(\rho_t\mmid\pi)^2\\
    &= e^{(q-1)R_q(\rho_t\mmid\pi)} - e^{(q-2)R_{q/2}(\rho_t\mmid\pi)}\\
    &\geq  e^{(q-1)R_q(\rho_t\mmid\pi)} - e^{(q-2)R_q(\rho_t\mmid\pi)}\\
    &= F_q(\rho_t\mmid\pi)\left(1-e^{-R_q(\rho_t\mmid\pi)}\right),
\end{align*}
where we used the fact that $q \mapsto R_q(\rho_t\mmid\pi)$ is non-decreasing.
\end{proof}

\begin{proof}[Proof of Theorem \ref{thm:diffusion_conv}.]
Suppose $\pi$ satisfies \eqref{eq:wpi} with $\FIc_\WPI$ and $\Phi'$ and assume that $\Phi'\big(\big(\tfrac{\rho_t}{\pi}\big)^{q/2}\big) \leq \delta_0$ for all $t \in \R_+$. Choosing $f = \left(\frac{\rho_t}{\pi}\right)^{q/2}$, it follows from \eqref{eq:wpi} and Lemma \ref{lem:quantity_estimates} that
\begin{align*}
    \frac{q^2G_q(\rho_t\mmid\pi)}{4F_q(\rho_t\mmid\pi)} &\geq \frac{1-e^{-R_q(\rho_t\mmid\pi)}}{\FIc_\WPI(r)} -\frac{r\Phi(f)}{\FIc_\WPI(r)F_q(\rho_t\mmid\pi)}\\
    &\geq \frac{1-e^{-R_q(\rho_t\mmid\pi)}}{\FIc_\WPI(r)}-\frac{r\delta_0}{\FIc_\WPI(r)}.
\end{align*}
Hence,
\begin{equation*}
    \pderiv{R_q(\rho_t\mmid\pi)}{t} \leq \frac{-4(1-e^{-R_q(\rho_t\mmid\pi)}) + 4r\delta_0}{q\FIc_\WPI(r)}.
\end{equation*}
Thus with $R_q(\rho_t\mmid\pi) \geq 1$ we have
\begin{equation*}
    \pderiv{R_q(\rho_t\mmid\pi)}{t} \leq \frac{-2 + 4r\delta_0}{q\FIc_\WPI(r)},
\end{equation*}
and for $R_q(\rho_t\mmid\pi) < 1$ we have
\begin{equation*}
    \pderiv{R_q(\rho_t\mmid\pi)}{t} \leq \frac{-2R_q(\rho_t\mmid\pi) + 4r\delta_0}{q\FIc_\WPI(r)}.
\end{equation*}
Integration and Gr\"onwall's lemma yield
\begin{equation*}
    R_q(\rho_t\mmid\pi) \leq \begin{cases}
    R_q(\rho_0\mmid\pi) - \frac{2-4r\delta_0}{\FIc_\WPI(r)q}t, & \text{if} \quad R_q(\rho_0\mmid\pi),R_q(\rho_t\mmid\pi) \geq 1\\
    e^{-\tfrac{2t}{\FIc_\WPI(r)q}}(R_q(\rho_0\mmid\pi)-2r\delta_0) + 2r\delta_0, & \text{if} \quad R_q(\rho_0\mmid\pi) < 1.
    \end{cases}
\end{equation*}

Suppose that $q'=\infty$. In this case, we can choose the original WPI with $\Phi(\cdot) = \osc(\cdot)^2$. Then, we need to show $\osc\big(\big(\tfrac{\rho_t}{\pi}\big)^{q/2}\big)^2 \leq \norm{\tfrac{\rho_0}{\pi}}^q_{L^\infty(\pi)}$. Notice that $P_t\frac{\rho_0}{\pi} = \frac{\rho_t}{\pi}$ (which one can verify by the tower property of the conditional expectation and the symmetry of the Markov process). Then we have,
$$\osc\left(\left(P_t\frac{\rho_0}{\pi}\right)^{q/2}\right)^2 \leq \norm{P_t\frac{\rho_0}{\pi}}^q_{L^\infty(\pi)} \leq \norm{\frac{\rho_0}{\pi}}^q_{L^\infty(\pi)},$$
where the second inequality follows from the contraction property of the semigroup. Thus, we can set $\delta_0 = \norm{\tfrac{\rho_0}{\pi}}^q_{L^\infty(\pi)} = \exp(qR_\infty(\rho_0\mmid\pi))$, which completes the proof for the case of $q'=\infty$.

Now, suppose $q < q' < \infty$. Using $u = \tfrac{2q'}{q}$ in Proporostion \ref{prop:wpi_improved}, $\pi$ satisfies a WPI with weighting $\beta$ and $\Phi'(\cdot) = \norm{\cdot - \Epi{\cdot}}^2_{L^{2q'/q}(\pi)}$, where $\beta$ is given in the statement of Theorem \ref{thm:diffusion_conv}. Additionally, by Lemma \ref{lem:raw_central_moment_compare}, we have
$$\Phi'\left(\left(\frac{\rho_t}{\pi}\right)^{q/2}\right) \leq \norm{\left(\frac{\rho_t}{\pi}\right)^{q/2}}^2_{L^{2q'/q}(\pi)} = \norm{P_t\frac{\rho_0}{\pi}}^q_{L^{q'}(\pi)} \leq \norm{\frac{\rho_0}{\pi}}^q_{L^{q'}(\pi)} \leq \exp(qR_{q'}(\rho_0\mmid\pi)).$$
Thus in this case, we can choose $\delta_0 = \exp(qR_{q'}(\rho_0\mmid\pi))$, which finishes the proof.
\end{proof}

\subsection{Proof of Proposition \ref{prop:wpi_improved}}
\label{app:proof_prop_wpi_improved}
By Theorem 2.3 of \cite{Rockner2001-zs}, in order to show that \eqref{eq:wpi} holds with $\Phi'$ as in \eqref{eq:wpi_improved}, it suffices to show that for every mean-zero $f \in \mathcal{C}^\infty(\R^d)$,
\begin{equation}\label{eq:Lu_decay}
    \norm{P_tf}^2_{L^2(\pi)} \leq \xi(t)\norm{f}_{L^u(\pi)}^2
\end{equation}
such that $\xi : \reals_+ \to \reals_+$ is decreasing and $\lim_{t\to\infty}\xi(t) = 0$. Then, we obtain the \eqref{eq:wpi} with
$$\FIc'_\WPI(r) = 2r\inf_s \frac{1}{s}\xi^{-1}(s\exp(1-s/r)).$$
To establish \eqref{eq:Lu_decay}, a similar argument to that of~\cite[Lemma 5.1]{cattiaux2011central} shows that given such a function $f$, there exists a constant $K > 0$ such that,
$$\norm{P_t f}_{L^2(\pi)}^2 \leq 4^{1+2/u} \norm{f}_{L^u(\pi)}\left(\norm{P_t (\tilde{f} - \pi(\tilde{f}))}_{L^2(\pi)}^2 \right)^{1-2/u},$$
where $\tilde{f} = K^{-1} (f \wedge K \vee (-K))$. Since $f \in \mathcal{C}^\infty(\R^d)$, $\tilde{f}$ must be infinitely differentiable almost everywhere and thus, in the domain of the generator. Therefore, we may apply Theorem 2.1 of \cite{Rockner2001-zs} to deduce that
$$\norm{P_t(\tilde{f} - \pi(\tilde{f}))}_{L^2(\pi)}^2 \leq \lambda(t) (\Phi(\tilde{f}) + \operatorname{Var}_{\pi}(\tilde{f})) \leq 5\lambda(t),$$
where
$$\lambda(t) = \inf\{r > 0 \; : \; \frac{1}{2}\FIc_\WPI(r)\log(1/r) \leq t\}.$$
Hence, \eqref{eq:Lu_decay} holds with $\xi(t) = 5\lambda(t)^{1-2/u}$. By definition, $\lambda$ and hence $\xi$ are decreasing, and $\lim_{t\to\infty}\xi(t) = 0$. From this, we observe that
$$\FIc'_\WPI(r) \leq 2\xi^{-1}(r) = 2\lambda^{-1}\left((r/5)^{\tfrac{u}{u-2}}\right) = \FIc_\WPI\left((r/5)^{\tfrac{u}{u-2}}\right)\log\left((5/r)^{\tfrac{u}{u-2}} \lor 1 \right).$$
Finally, suppose $\pi$ does not satisfy a Poincar\'e inequality, but satisfies \eqref{eq:wpi_improved} for some $\FIc'_\WPI$ with $u=2$.
Then, for any $r > 0$ we have
$$(1-r)\var(f) \leq \beta_\WPI(r)\Earg{\norm{\nabla f}^2}.$$
Thus, $\pi$ satisfies a Poincar\'e inequality (with a constant at most $2\FIc_\WPI(1/2)$), which is a contradiction.
\qed

\subsection{Proof of Theorem \ref{thm:lmc_upper}}\label{app:proof_lmc_upper}
First, we present the following lemma which enables us to control the discretization error $R_{2q}(\mu_N\mmid\rho_T)$. This proposition can be retrieved by a careful evaluation of the terms in the proof of Proposition 25 of \cite{Chewi2021-vj}. We avoid the simplifications made by \cite{Chewi2021-vj} since that can affect our rate due to $m = \mathcal{O}(d^{1/\alpha})$ for $\pi_\alpha$ defined in Section \ref{sec:conv} with $\alpha \in (0,1)$, and also $T$ can be exponential in $d$ for specific examples.
\begin{proposition}[{\cite[Proposition 25]{Chewi2021-vj}}]
\label{prop:disc_error}
Suppose $\nabla V$ is $s$-H\"older continuous with constant $L$ and $\nabla V(0) = 0$. Let $m \coloneqq \tfrac{1}{2}\inf\{R \, : \, \pi(\norm{x} \geq R) \leq \tfrac{1}{2}\}$. Define 
\begin{equation}\label{eq:modified-target}
\hat{\pi}(x) \propto \exp\left(-V(x) - \frac{1}{6144T}\max\{\norm{x}-2m,0\}^2\right).
\end{equation}
Assume for simplicity $\varepsilon^{-1},m,L,T,R_2(\rho_0\mmid\hat{\pi})\geq 1$, Then, for $q \leq \varepsilon^{-1}$, if the step size satisfies
\begin{equation*}
    h \leq \mathcal{O}_s\left(\frac{\varepsilon^{1/s}}{dq^{1/s}L^{2/s}T^{1/s}}\min\left\{1,\frac{L^{1/s-1}T^{1/(2s)}d}{\varepsilon^{1/(2s)}m^s},\frac{L^{1/s-1}T^{(1-s^2)/(2s)}d}{\varepsilon^{1/(2s)}R_2(\rho_0\mmid\hat{\pi})^{s/2}\ln(N)^{s/2}}\right\}\right),
\end{equation*}
we have for $T=Nh$
\begin{equation*}
    R_q(\mu_N\mmid\rho_T) \leq \varepsilon.
\end{equation*}
\end{proposition}
\begin{remark}
    If the first moment of $\pi$ is finite, we have $m \leq \Epi{\norm{x}}$ by Markov's inequality. In fact, the original result in \cite{Chewi2021-vj} is presented using $m = \Epi{\norm{x}}$. However, the result is still valid with the choice of $m$ in Proposition \ref{prop:disc_error}, which is useful for targets with infinite first moment.
\end{remark}
With this proposition in hand, we are ready to state the proof of Theorem \ref{thm:lmc_upper}.

\begin{proof}[Proof of Theorem \ref{thm:lmc_upper}]
By Theorem \ref{thm:diffusion_conv} we have $R_{2q-1}(\rho_T\mmid\pi) \leq \tfrac{\varepsilon}{2}$. Furthermore, by choosing the step-size as in Proposition \ref{prop:disc_error}, we have $R_{2q}(\mu_N\mmid\rho_T) \leq \frac{\varepsilon}{2}$. By the weak triangle inequality of R\'enyi divergence, we have $R_q(\mu_N\mmid\pi) \leq \varepsilon$ for
$$N = \frac{T}{h} = \Theta_s\left(\frac{dq^{1/s}L^{2/s}T^{1+1/s}}{\varepsilon^{1/s}}\max\left\{1, \frac{\varepsilon^{1/(2s)}m^s}{L^{1/s-1}T^{1/(2s)}d},\frac{\varepsilon^{1/(2s)}R_2(\rho_0\mmid\hat{\pi})^{s/2}\ln(N)^{s/2}}{L^{1/s-1}T^{(1-s^2)/(2s)}d}\right\}\right).$$
A further simplification of this result yields the statement of the Theorem.
\end{proof}

\subsection{Computing Weak Poincar\'e Constants}\label{app:const}
In this section, we will provide WPI estimates for our model examples $\pi_\nu(x) \propto (1 + \norm{x}^2)^{\tfrac{d+\nu}{2}}$ and $\pi_\alpha(x) \propto \exp(-(1 + \norm{x}^2)^{\alpha/2})$.

We will use the following chain of implications to establish WPIs with suitable dimension dependencies:
$$
\begin{array}{ccccc}
   \text{Weighted Poincar\'e}  & \Rightarrow & \text{Converse Poincar\'e} & \Rightarrow & \text{Weak Poincar\'e}\\
\end{array}
$$
In particular, to obtain a WPI from a converse Poinacr\'e inequality, we can use the following result due to \cite{cattiaux2010functional}.
\begin{lemma}[Theorem 5.1 of \cite{cattiaux2010functional}]
    \label{lem:conv_to_weak_poincare}
    Assume $\pi$ satisfies a converse Poincar\'e inequality, i.e.\
    $$\inf_c \int\abs{f(x) - c}^2w(x)d\pi(x) \leq C\int\norm{\nabla f(x)}^2d\pi(x),$$
    for some non-negative weight function $w$, such that $\int wd\pi < \infty$. Define $G(r) \coloneqq \inf\{u : \pi(w \leq u) > r\}$. Then, $\pi$ satisfies a WPI with $\Phi(\cdot) = \osc(\cdot)^2$ and $\FIc_\WPI(r) = \tfrac{C}{G(r)}$.
\end{lemma}
Hence, our main effort is to establish a converse Poincar\'e inequality for our model examples. In fact, for generalized Cauchy measures, we can immediately use a result of \cite{bobkov2009weighted}.
\begin{lemma}[Corollary 3.2 of \cite{bobkov2009weighted}]
    \label{lem:conv_poincare_cauchy}
    Let $\pi_\nu(x) \propto (1+\norm{x}^2)^{-(d+\nu)/2}$ with $p > 0$. Then, for any $f \in \mathcal{C}^\infty(\R^d)$, we have the following converse (weighted) Poincar\'e inequality
    \begin{equation}
        \inf_c \int \abs{f(x) - c}^2w(x)d\pi_\nu(x) \leq C_{d,\nu}\int \norm{\nabla f(x)}^2d\pi_\nu(x),
    \end{equation}
    for $w(x) = \frac{1}{1+\norm{x}^2}$, with
    $$C_{d,\nu} \coloneqq \begin{cases}
    \frac{1}{d+\nu} & \text{if} \, \nu \geq d + 2\\
    \frac{2}{\nu} & \text{otherwise}
    \end{cases}.$$
\end{lemma}
Thus, along with a concentration bound, we can invoke Lemma \ref{lem:conv_to_weak_poincare} to estimate $\FIc_\WPI$ for generalized Cauchy measures.

\begin{proof}[Proof of Proposition \ref{prop:wpi_t}]
    In order to invoke Lemma \ref{lem:conv_to_weak_poincare}, we need to estimate $\pi(w \leq u)$ for the weight function given by Lemma \ref{lem:conv_poincare_cauchy} $w(x) = (1+\norm{x}^2)^{-1}$. By Lemma \ref{lem:student_t_tail}, we have 
    $$\pi(w \leq u) = \pi\left(\norm{x}\geq \sqrt{u^{-1}-1}\right) \leq (d+\nu)^{\nu/2}(u^{-1}-1)^{-\nu/2}.$$
    Choosing $u^{-1} = 1 + (d+\nu)r^{-2/\nu}$, we obtain by invoking Lemma \ref{lem:conv_to_weak_poincare}
    $$\FIc_\WPI(r) = \frac{2}{\nu} + 2\left(\frac{d}{\nu} + 1\right)r^{-2/\nu}.$$
\end{proof}

Estimating $\FIc_\WPI$ for $\pi_\alpha$ is more involved as we do not readily have a suitable converse Poincar\'e inequality. We will work towards this by first deriving a weighted Poincar\'e inequality using the perturbation argument.

\begin{lemma}\label{lem:weighted PI subexp} Let $\pi_\alpha(x) \propto \exp\left( -\left(1+\norm{x}^2\right)^{\alpha/2} \right)$ with $\alpha\in (0,1)$. Then for any $f \in \mathcal{C}^\infty(\R^d)$, we have the following weighted Poincar\'{e} inequality
\begin{align}\label{eq:weighted PI subexp}
   \var_{\pi_\alpha}(f)\le e C_{d,\alpha} \int w(x)^2 \norm{ \nabla f(x)}^2  d\pi_\alpha(x).    
\end{align}
for $w(x)=\norm{x}^{1-\alpha}$ with $C_{d,\alpha}$ satisfying \eqref{eq:norm density WPI subexp}.
\end{lemma}
\begin{proof} Let $\Tilde{\pi}_\alpha(x)\propto \exp\left( -\norm{x}^\alpha \right)$. According to \cite[Proposition 4.7]{cattiaux2010functional}, $\Tilde{\pi}_\alpha$ satisfies a weighted Poincar\'{e} inequality with weight $w(x)^2=\norm{x}^{2(1-\alpha)}$ and parameter $C_{d,\alpha}$ satisfies 
\begin{align}\label{eq:norm density WPI subexp}
    \frac{d}{\alpha^3}\le C_{d,\alpha} \le 12\frac{d}{\alpha^3}+\frac{d+\alpha}{\alpha^4}
\end{align}
Meanwhile, $\frac{d\pi_\alpha}{d\tilde{\pi}_\alpha} = \exp\left( k(x) \right)$ with $k(x)=-\left(1+\norm{x}^2\right)^{\alpha/2} +\norm{x}^\alpha+\text{constant}$. Notice that $\osc(k) = 1$. Then, \eqref{eq:weighted PI subexp} follows from the perturbation property of the weighted Poincar\'e inequality.
\end{proof}

In the next step, by an argument similar to \cite[Proposition 3.3]{bobkov2009weighted}, we transform the weighted Poincar\'e inequality for $\pi_\alpha$ to a converse Poincar\'e inequality.
\begin{lemma}\label{lem:lem:converse weighted Poincare subexp true weights} 
    Let $\pi_\alpha(x) \propto \exp\left( -\left(1+\norm{x}^2\right)^{\alpha/2} \right)$ with $\alpha\in (0,1)$. Let $\gamma\in(0,2\alpha]$. Then for any $g \in \mathcal{C}^\infty(\R^d)$, we have the following converse weighted Poincar\'{e} inequality
    \begin{align}\label{eq:converse weighted PI subexp true weights}
        \inf_{c} \int \frac{\left| f(x)-c \right|^2}{ \left( a+b\norm{x}^2 \right)^{1-\alpha+\gamma/2}} d\pi_\alpha(x)\le \left[ 1-\left(1-\alpha+\gamma/2\right)b^{\frac{1}{2}}a^{-\frac{1}{2}(\alpha-\gamma/2)} \right]^{-2} \int \norm{\nabla{f}(x)}^2 d\pi_\alpha(x)
    \end{align}
with $a=\frac{\gamma(eC_{d,\alpha})^{\frac{2}{\gamma}}}{2(1-\alpha)+\gamma}$, $b=\frac{2(1-\alpha)}{2(1-\alpha)+\gamma}$, and $C_{d,\alpha}$ satisfies \eqref{eq:norm density WPI subexp}
\end{lemma}
\begin{proof}
    First we apply Young's inequality to bound the weights in Lemma \ref{lem:weighted PI subexp}.
    \begin{align*}
        eC_{d,\alpha} \norm{x}^{2(1-\alpha)}&=\left[ (eC_{d,\alpha})^{\frac{2}{2(1-\alpha)+\gamma}}\norm{x}^{\frac{4(1-\alpha)}{2(1-\alpha)+\gamma}} \right]^{1-\alpha+\gamma/2}\\
        &\le \left[ \frac{(eC_{d,\alpha})^{\frac{2}{\gamma}}}{\frac{2(1-\alpha)+\gamma}{\gamma}} + \frac{\norm{x}^2}{\frac{2(1-\alpha)+\gamma}{2(1-\alpha)}} \right]^{1-\alpha+\gamma/2}\\
        &\eqqcolon\left( a+b\norm{x}^2 \right)^{1-\alpha+\gamma/2}
    \end{align*}
    Therefore, Lemma \ref{lem:weighted PI subexp} yields
    \begin{align}\label{eq:weighted PI intermediate}
        \var_{\pi_\alpha}(g)\le \int \norm{\nabla g(x)}^2 \left(a+b\norm{x}^2\right)^{1-\alpha+\gamma/2}d\pi_\alpha(x)
    \end{align}
    The rest of the proof proceeds similarly to that of \cite[Proposition 3.3]{bobkov2009weighted}. Consider $g(x)=(f(x)-c)\left(a+b\norm{x}^2\right)^{-\frac{1-\alpha+\gamma/2}{2}}$, where $c$ is chosen such that $g$ has mean $0$. Then we have
    \begin{align*}
        \nabla g(x)=\nabla f(x)\left(a+b\norm{x}^2\right)^{-\frac{1-\alpha+\gamma/2}{2}} -b\left( 1-\alpha+\gamma/2 \right)\left(a+b\norm{x}^2\right)^{-\frac{1-\alpha+\gamma/2}{2}-1}(f(x)-c)x.
    \end{align*}
    ‌By the elementary inequality $\norm{u + v}^2 \leq \tfrac{r}{r-1}\norm{u}^2 + r\norm{v}^2$ for any $r>1$, we have
    \begin{align*}
        \norm{\nabla g(x)}^2&\le \frac{r}{r-1} \norm{\nabla f(x)}^2 \left(a+b\norm{x}^2\right)^{-(1-\alpha+\gamma/2)}\\
        &\quad +rb^2 \left( 1-\alpha+\gamma/2 \right)^2 \frac{\norm{x}^2}{a+b\norm{x}^2}\frac{1}{\left(a+b\norm{x}^2\right)^{\alpha-\gamma/2}}\left(a+b\norm{x}^2\right)^{-2(1-\alpha+\gamma/2)}(f(x)-c)^2\\
        &\le \frac{r}{r-1} \norm{\nabla f(x)}^2 \left(a+b\norm{x}^2\right)^{-(1-\alpha+\gamma/2)}\\
        &\quad +rb^2 \left( 1-\alpha+\gamma/2 \right)^2 b^{-1}a^{-\alpha+\gamma/2}\left(a+b\norm{x}^2\right)^{-2(1-\alpha+\gamma/2)}(f(x)-c)^2
    \end{align*}
    Applying \eqref{eq:weighted PI intermediate} to $g$, we obtain
    \begin{align*}
        \int (f(x)-c)^2 &\left(a+b\norm{x}^2\right)^{-(1-\alpha+\gamma/2)} d\pi_\alpha(x)\le \frac{r}{r-1} \int \norm{\nabla f(x)}^2 d\pi_\alpha(x)\\
        &\quad+rba^{-\alpha+\gamma/2}\left( 1-\alpha+\gamma/2 \right)^2 \int (f(x)-c)^2 \left(a+b\norm{x}^2\right)^{-(1-\alpha+\gamma/2)} d\pi_\alpha(x)
    \end{align*}
    Since $\gamma\in (0,2\alpha]$, $b\in (0,1)$, $a>1$ and $1-\alpha+\gamma/2\in (1-\alpha,1)$, we have
    \begin{align*}
        \int f(x)^2 \left(a+b\norm{x}^2\right)^{-(1-\alpha+\gamma/2)} d\pi_\alpha(x)&\le \frac{r}{r-1} \frac{1}{1-rba^{-\alpha+\gamma/2}\left( 1-\alpha+\gamma/2 \right)^2} \int \norm{\nabla f(x)}^2 d\pi_\alpha(x)
    \end{align*}
    Last, \eqref{eq:converse weighted PI subexp true weights} follows by choosing the optimal $r^*=b^{-\frac{1}{2}}a^{\frac{1}{2}(\alpha-\gamma/2)}(1-\alpha+\gamma/2)^{-1}>1$.
\end{proof}

With this converse Poincar\'e inequality, we are ready to invoke Lemma \ref{lem:conv_to_weak_poincare} and obtain a WPI for $\pi_\alpha$.

\begin{proof}[Proof of Proposition \ref{prop:WPI subexp true weights}] Let $a$ and $b$ be defined according to the statement of Lemma \ref{lem:lem:converse weighted Poincare subexp true weights}. In order to invoke Lemma \ref{lem:conv_to_weak_poincare}, we need to estimate $\pi_\alpha(w \leq u)$ for the weight $w(x)=\left(a+b\norm{x}^2\right)^{-(1-\alpha + \gamma/2)}$. Using the tail bound of Lemma \ref{lem:tail bound subexp}, ƒ
    \begin{align*}
        \pi_{\alpha}\left(w(x)\le u \right)&=\pi_\alpha\left( \norm{x}\ge b^{-\frac{1}{2}}\sqrt{u^{-\frac{1}{1-\alpha+\gamma/2}}-a} \right)\\
        &\le  e^{\frac{1}{2}} 2^{d/\alpha} \exp\left( -\frac{1}{2}\left( 1+b^{-1}\left( u^{-\frac{1}{1-\alpha+\gamma/2}}-a \right) \right)^{\alpha/2} \right).
    \end{align*}
    Choosing $u$ such that the above is at most $r$, we obtain
    \begin{align*}
        G(r)&:=\inf\left\{ u:\ \pi_\alpha\left( w(x)\le u \right)>r \right\}\\
        &\ge \left\{ a+b\left[ \left( 1+\frac{2d}{\alpha}\ln 2 +2\ln\left(r^{-1}\right) \right)^{\frac{2}{\alpha}}-1 \right] \right\}^{-(1-\alpha+\gamma/2)}\\
        &\ge \left\{ a+b \left(1 + \frac{2d}{\alpha}\ln 2 +2\ln\left(r^{-1}\right) \right)^{\frac{2}{\alpha}} \right\}^{-(1-\alpha+\gamma/2)}.
    \end{align*}
    Therefore, by Lemma \ref{lem:conv_to_weak_poincare} the weak Poincar\'{e} constant satisfies
    \begin{align*}
        \alpha_{\WPI}(r)& \le   \left[ 1-\left(1-\alpha+\gamma/2\right)b^{\frac{1}{2}}a^{-\frac{1}{2}(\alpha-\gamma/2)} \right]^{-2} \left\{ a+b \left(1 + \frac{2d}{\alpha}\ln 2 +2\ln\left(r^{-1}\right) \right)^{\frac{2}{\alpha}} \right\}^{1-\alpha+\gamma/2}\\
        &\le \left( 1-a^{-\frac{1}{2}(\alpha-\gamma/2)} \right)^{-2} \left( a^{1-\alpha+\gamma/2} +\left(1 + \frac{2d}{\alpha}\ln 2+2\ln (r^{-1}) \right) ^{\frac{2(1-\alpha)+\gamma}{\alpha}}\right)
    \end{align*}
    Recall $a=\frac{\gamma}{2(1-\alpha)+\gamma}(eC_{d,\alpha})^{\frac{2}{\gamma}}<(eC_{d,\alpha})^{\frac{2}{\gamma}}$, $\gamma \in (0,2\alpha]$, and notice that $\inf_\gamma a > 1$. Therefore,
    \begin{align*}
        \alpha_{\WPI}(r)&\le \frac{3^{\tfrac{2-3\alpha+\gamma}{\alpha}}\lor 1}{\left( 1-a^{-\tfrac{1}{2}(\alpha-\gamma/2)}\right)^2}\left( (eC_{d,\alpha})^{\frac{2(2-2\alpha+\gamma)}{\gamma}}+ 1 + \left( \frac{2\ln 2}{\alpha} \right)^{\frac{2-2\alpha+\gamma}{\alpha}} d^{\frac{2-2\alpha+\gamma}{\alpha}}+ 2^{\frac{2-2\alpha+\gamma}{\alpha}} \ln\left(r^{-1}\right)^{\frac{2-2\alpha+\gamma}{\alpha}} \right)\\
        &\le C_\alpha \left( d^{\frac{2(2-2\alpha+\gamma)}{\gamma}}+\ln\left(r^{-1}\right)^{\frac{2-2\alpha+\gamma}{\alpha}}\right)
    \end{align*}
\end{proof}

%% file: appendices/proofs_lower_bounds.tex
\label{app:lower_methodology}
Our goal in this section is to prove Theorem \ref{thm:lower_bound} and Proposition \ref{prop:step_size_bound_summary}. To do so, we begin by introducing a lower bound that compares the second moment of $\rho$ with a certain moment from $\pi$ in order to lower bound the R\'enyi divergence between $\rho$ and $\pi$, thus allowing us to track the evolution of the second moment of the process rather than the R\'enyi divergence itself.
\begin{lemma}\label{lem:wass_comparison}
Let $q > 1$, and $\rho$ and $\pi$ be a pair of measures with $\rho(\norm{\cdot}^2) < \infty$ and $\pi(\norm{\cdot}^{\tfrac{2q}{q-1}}) < \infty$, then the q-R\'enyi divergence is lower bounded by
    \begin{equation*}
        R_q(\rho \| \pi) \geq \ln \bigg ( \frac{\rho(\|\cdot\|^2)^{\frac{q}{q-1}}}{\pi(\|\cdot\|^{\frac{2q}{q - 1}})} \bigg ).
    \end{equation*}
\end{lemma}
\begin{proof}
From \cite[Theorem 3.1]{Birrell2020-gu}, we have the following variational representation for the R\'enyi divergence,
$$R_q(\rho\mmid\pi) \geq \sup_\phi \left\{\frac{q}{q-1}\ln\int\exp((q-1)\phi(x))d\rho(x) - \ln\int\exp(q\phi(x))d\pi(x)\right\},$$
where the supremum runs over all measurable functions.
The choice of \(\phi(x) = \frac{2}{q-1}\ln(\|x\|)\) proves the statement of the lemma.
\end{proof}

The following lemma uses the gradient bound condition in \eqref{eq:grad_v_growth} to lower bound the decay rate of the second moment for the Langevin diffusion and LMC.
\begin{lemma}[Evolution of the Second Moment]\label{lem:lang_gf_comparison}
    Suppose Eq.~\eqref{eq:grad_v_growth} holds with $\alpha \in [0,2]$, and $\E\norm{X_0}^2 < \infty$. Then,
    $$\frac{\D}{\D t}\E\norm{X_t}^2 \geq 2d - 2b\Earg{\norm{X_t}^2}^{\alpha/2}.$$
    Similarly, LMC satisfies
    $$\E\norm{x_{k+1}}^2 \geq \E\norm{x_k}^2 - 2bh\Earg{\norm{x_k}^2}^{\alpha/2} + 2hd.$$
\end{lemma}
\begin{proof}
We begin by proving the result for LMC. By the independence of $\xi_k$ and $x_k$, we have
\begin{align*}
    \E\norm{x_{k+1}}^2 &= \E\norm{x_k - h\nabla V(x_k)}^2 + 2hd\\
    &\geq \E\norm{x_k}^2 - 2h\E\binner{\nabla V(x_k)}{x_k} + 2hd\\
    &\geq \E\norm{x_k}^2 - 2bh\Earg{\norm{x_k}^2}^{\alpha/2} + 2hd,
\end{align*}
where the last inequality follows from Cauchy-Schwartz and Jensen's inequalities. 

For the diffusion, it follows from It\^{o}'s lemma that
\begin{equation}\label{eq:ito_lemma_interm}
    \norm{X_t}^2 = -2\int_0^t \binner{\nabla V(X_s)}{X_s}\D s + 2td + 2\sqrt{2} \int_0^t \langle X_s, \D B_s \rangle.
\end{equation}
We proceed by showing the last term is a martingale and can be removed once taking expectations. For this, it is sufficient to show that \(X_t\) is \(B_t\)-integrable or equivalently,
\begin{equation}\label{eq:integ_cond}
    \E \Big [ \int_0^t \|X_s\|^2 \D s \Big ] < \infty.
\end{equation}
Indeed, from It\^o's lemma, and Tonelli's theorem,
\begin{align*}
    \E\norm{X_t}^2 &\leq -2\Earg{\int_0^t \binner{\nabla V(X_s)}{X_s}\D s} + 2td + 2\sqrt{2}\Earg{\left(\int_0^t\binner{X_s}{\D B_s}\right)^2}^{1/2}\\
    &\leq 2(b+d)t + 2b\Earg{\int_0^t\norm{X_s}^2\D s} + 2\sqrt{2}\Earg{\int_0^t\norm{X_s}^2\D s}^{1/2}\\
    &\leq \sqrt{2} + 2(b+d)t + (2b + \sqrt{2})\int_0^t\E \norm{X_s}^2 \D s.
\end{align*}
Define $f(s) \coloneqq \sup_{r \in [0,s]}\Earg{\norm{X_r}^2}$. Then, for any $T \geq t$,
$$f(t) \leq \underbrace{\sqrt{2} + 2(b+d)T}_{\eqqcolon C_1} + \underbrace{(2b + \sqrt{2})}_{\eqqcolon C_2}\int_0^tf(s)\D s.$$
Then, by Gr\"onwall's lemma,
$$\int_0^tf(s)ds \leq \frac{C_1}{C_2}(\exp(C_2t) - 1) < \infty,$$
and consequently \eqref{eq:integ_cond} is satisfied. Thus, we can remove the last term in \eqref{eq:ito_lemma_interm} after taking expectation, and after taking a time derivative obtain
\begin{align*}
    \frac{\D}{\D t}\E\norm{X_t}^2 &= -2\Earg{\binner{\nabla V(X_t)}{X_t}} + 2d\\
    &\geq -2b\E\norm{X_t}^\alpha + 2d\\
    &\geq -2b\Earg{\norm{X_t}^2}^{\alpha/2} + 2d,
\end{align*}
where once again the last inequality is implied by Cauchy-Schwartz and Jensen's inequalities.
\end{proof}

Another key ingredient of our proof will be controlling the R\'enyi or KL divergence using the variance of an isotropic Gaussian initialization, which is provided by the following lemmas.
\begin{lemma}[Controlling R\'enyi Divergence by Initial Variance]\label{lem:init_renyi}
Let $\pi \propto e^{-V}$, $Z \coloneqq \int e^{-V(x)}dx$ and suppose $V$ satisfies \eqref{eq:grad_v_growth}. Then the following holds:
\begin{enumerate}
    \item If $\alpha = 0$, then for $\sigma^2 \geq (d+\nu)^{-1}$, where we recall $\nu \coloneqq b - d > 0$,
$$R_\infty(\mathcal{N}_{\sigma^2 I_d}\mmid\pi) \leq \frac{\nu}{2}\ln\sigma^2 + \ln\left\{\left(\frac{Z}{(2\pi)^{d/2}}\right)\left(\frac{d + \nu}{e}\right)^{\tfrac{d+\nu}{2}}\right\} + \frac{1}{2\sigma^2}.$$
    \item If $\alpha \in (0, 2)$, then for $\sigma^2 \geq b^{-1}$,
    $$R_\infty(\mathcal{N}_{\sigma^2 I_d}\mmid\pi) \leq \frac{(b\sigma^2)^{\tfrac{2}{2-\alpha}}}{\alpha} + \ln\left(\frac{Z}{(2\pi\sigma^2)^{d/2}}\right) + \frac{1}{2\sigma^2}.$$
\end{enumerate}
\end{lemma}
\begin{proof}
We begin by upper bounding $V$. Using \eqref{eq:grad_v_growth}, we have that
\begin{equation*}
    V(x) \leq \int_0^1\norm{\nabla V(tx)}\norm{x}\D t  \leq\begin{cases}\frac{b}{\alpha}\left((1 + \norm{x}^2)^{\alpha/2} - 1\right) & \text{if} \quad \alpha > 0,\\
    \frac{b}{2}\ln(1 + \norm{x}^2) & \text{if} \quad \alpha = 0.\end{cases}
\end{equation*}
In the case that $\alpha = 0$, it follows that
$$R_\infty(\mathcal{N}_{\sigma^2 I_d}\mmid\pi) \leq \ln\left(\frac{Z}{(2\pi\sigma^2)^{d/2}}\right) + \sup_x \frac{-\norm{x}^2}{2\sigma^2} + \frac{b}{2}\ln(1 + \norm{x}^2).$$
For $\sigma^2 \geq \tfrac{1}{a + b}$, the supremum occurs at $\norm{x^*}^2 = b\sigma^2 - 1$, and
$$R_\infty(\mathcal{N}_{\sigma^2 I_d}\mmid\pi) \leq \frac{b-d}{2}\ln\sigma^2 + \ln\left(\frac{Z}{(2\pi)^{d/2}}\right) + \frac{1}{2\sigma^2} + \frac{b}{2}\ln\frac{b}{e}.$$
Now suppose $\alpha > 0$, then
$$R_\infty(\mathcal{N}_{\sigma^2 I_d}\mmid\pi) \leq \ln\left(\frac{Z}{(2\pi\sigma^2)^{d/2}}\right) + \sup_x \frac{-\norm{x}^2}{2\sigma^2} + \frac{b}{\alpha}(1 + \norm{x}^2)^{\alpha/2}.$$
For $\sigma^2 \geq b^{-1}$, the supremum is attained at $\norm{x^*}^2 = (b\sigma^2)^{\tfrac{2}{2-\alpha}} - 1$, and
\begin{align*}
    R_\infty(\mathcal{N}_{\sigma^2 I_d}\mmid\pi) \leq \frac{b^{\tfrac{2}{2-\alpha}}}{\alpha}\sigma^{\tfrac{2\alpha}{2-\alpha}} + \ln\left(\frac{Z}{(2\pi\sigma^2)^{d/2}}\right) + \frac{1}{2\sigma^2}.
\end{align*}
\end{proof}

We pause here to state the result of Lemma \ref{lem:init_renyi} more explicitly for the generalized Cauchy potential $V_\nu$ and the sublinear potential $V_\alpha$. For the first example, we have the normalizing constant
$$Z_\nu \coloneqq \frac{\pi^{d/2}\Gamma(\nu/2)}{\Gamma((d+\nu)/2)} \leq \pi^{d/2}\Gamma(\nu/2)\left(\frac{d+\nu}{2e}\right)^{-\tfrac{d+\nu}{2} + 1},$$
where the second inequality holds when $d \geq 2$ using $\Gamma(z) \geq (z/e)^{z-1}$ for $z \geq 1$.
\begin{corollary}\label{cor:gaussian_cauchy_renyi}
Consider the measure $\pi_\nu \propto \exp(-V_\nu)$ with $V_\nu(x) = \frac{d+\nu}{2}\ln(1 + \norm{x}^2)$. Then, for $d \geq 2$, $\alpha \in (0,2]$, and $\sigma^2 \geq (d + \nu)^{-1}$, we have
$$R_\infty(\mathcal{N}_{\sigma^2 I_d}\mmid\pi_\nu) \leq \frac{\nu}{2}\ln\sigma^2 + \ln\left(2^{\nu/2}\Gamma(\nu/2)\right) + \ln\left(\frac{d+\nu}{2e}\right).$$
\end{corollary}
The second example, the sub-linear potential $V_\alpha = (1 + \norm{x}^2)^{\alpha/2} - 1$ with $\alpha \in (0,1)$, satisfies \eqref{eq:grad_v_growth} with $\alpha$, and its normalizing constant can be estimated by
$$Z_\alpha \coloneqq \int \exp((1 + \norm{x}^2)^{\alpha/2}-1)dx \leq \int \exp(\norm{x}^\alpha)dx = \frac{\pi^{d/2}d/\alpha\Gamma(d/\alpha)}{\Gamma(d/2 + 1)} \leq \pi^{d/2}\left(\frac{d}{\alpha}\right)^{d/\alpha - d/2},$$
where we refer to \eqref{eq:sublin_normalizing_const} for a proof of the identity, and the second inequality follows from Lemma \ref{lem:gamma_ratio}.
\begin{corollary}\label{cor:gaussian_sublin_renyi}
    Consider the measure $\pi_\alpha \propto \exp(-V_\alpha)$ with $V_\alpha = (1 + \norm{x}^2)^{\alpha/2}$, $\alpha \in (0,1)$, and $d \geq 2$. Then for $\sigma^2 \geq 1/\alpha$,
    $$R_\infty(\mathcal{N}_{\sigma^2 I_d}\mmid\pi_\alpha) \leq (\alpha\sigma^2)^{\tfrac{\alpha}{2-\alpha}} + \left(\frac{d}{\alpha} - \frac{d}{2}\right)\ln(d/\alpha) - \frac{d}{2}\ln(2\sigma^2) + \frac{1}{2\sigma^2}.$$
\end{corollary}
When working with tail growth of order $\alpha = 2$, a large initial variance can lead to infinity R\'enyi divergence of any order $q > 1$. Hence, we will instead use the KL divergence as our initial metric in this setting.
\begin{lemma}[Controlling KL Divergence by Initial Variance]\label{lem:init_kl}
Let $Z \coloneqq \int e^{-V(x)}dx$ and suppose that $V$ satisfies \eqref{eq:grad_v_growth} with $\alpha = 2$, then $$\KL(\mathcal{N}_{\sigma^2 I_d}\mmid\pi) \leq \frac{(b\sigma^2 - 1)d}{2} + \ln\left(\frac{Z}{(2\pi\sigma^2)^{d/2}}\right).$$
\end{lemma}
\begin{proof}
    First, we upper bound $V$ with
    $$V(x) \leq \int_0^1\norm{\nabla V(tx)}\norm{x}\D t \leq \frac{b}{2}\norm{x}^2.$$
    Thus we have
    \begin{align*}
        \KL(\mathcal{N}_{\sigma^2 I_d}\mmid\pi) &= \E_{\mathcal{N}_{\sigma^2 I_d}}\left[\ln\left(\frac{Z}{(2\pi\sigma^2)^{d/2}}\exp\left(\frac{-\norm{x}^2}{2\sigma^2} + V(x)\right)\right)\right]\\
        &\leq \ln\left(\frac{Z}{(2\pi\sigma^2)^{d/2}}\right) + \frac{(b\sigma^2 - 1)d}{2}.
    \end{align*}
\end{proof}

We are now ready to present the proof of Theorem \ref{thm:lower_bound}.

\begin{proof}[Proof of Theorem \ref{thm:lower_bound}]
    Notice that to lower bound the time or iteration complexity, it suffices to lower bound the time or iteration complexity for one specific initialization with initial divergence less than $\Delta_0$. Our general strategy will be to use Gaussian initializations with large initial variances, such that Lemma \ref{lem:init_renyi} ensures the initial divergence is less than $\Delta_0$, while the time estimate from Lemma \ref{lem:lang_gf_comparison} provides the lower bound. Throughout this proof, $Z \coloneqq \int e^{-V}$ denotes the normalizing constant.
    \begin{enumerate}
        \item \underline{The case $\alpha = 0$}: Suppose that $\Delta_0$ satisfies
        \begin{equation}\label{eq:delta0_lower_cauchy}
            \Delta_0 \geq \left\{1 + 2\ln\left(\frac{Z((d+\nu)/e)^{\tfrac{d+\nu}{2}}}{(2\pi)^{d/2}}\right)\right\}\lor\nu\ln\left(\frac{2e\pi(\norm{\cdot}^{\tfrac{2q}{q-1}})^{\tfrac{q-1}{q}}}{d}\right)\lor\nu.
        \end{equation}
        Choose $\rho_0 = \mathcal{N}(0,\sigma^2 I_d)$ with $\sigma^2 = \exp\left(\frac{\Delta_0}{\nu}\right)$. By Lemma \ref{lem:init_renyi} we have
        $$R_\infty(\rho_0\mmid\pi) \leq \frac{\Delta_0}{2} + \ln\left\{\left(\frac{Z}{(2\pi)^{d/2}}\right)\left(\frac{d+\nu}{e}\right)^{\tfrac{d+\nu}{2}}\right\} + \frac{1}{2\sigma^2} \leq \Delta_0.$$
        Note that that due to Lemma \ref{lem:wass_comparison}, in order to have $R_q(\rho_T\mmid\pi) \leq 1$, $T$ needs to be sufficiently large such that
        $$\Earg{\norm{X_T}^2} \leq e^{\frac{q-1}{q}}\pi(\norm{\cdot}^{\tfrac{2q}{q-1}})^{\tfrac{q-1}{q}}.$$
        From Lemma \ref{lem:lang_gf_comparison} we obtain
        $$T \geq \frac{\Earg{\norm{X_0}^2} - \Earg{\norm{X_T}^2}}{2\nu},$$
        and
        $$N \geq \frac{\Earg{\norm{X_0}^2} - \Earg{\norm{X_k}^2}}{2h\nu}.$$
        Consequently, $T$ needs to satisfy
        $$T \geq \frac{d\exp\left(\frac{\Delta_0}{\nu}\right) - e\pi(\norm{\cdot}^{\tfrac{2q}{q-1}})^{\tfrac{q-1}{q}}}{2\nu} \geq \frac{d\exp\left(\frac{\Delta_0}{\nu}\right)}{4\nu},$$
        and $N$ needs to satisfy
        $$N \geq \frac{d\exp\left(\frac{\Delta_0}{\nu}\right)}{4h\nu}.$$

        \item \underline{The case $0 < \alpha < 2$}: 
        Suppose that $\Delta_0$ satisfies
        \begin{equation}\label{eq:delta0_lower_sublin}
            \Delta_0 \geq \left\{\frac{b^{\tfrac{\alpha}{2-\alpha}}}{\alpha}\left(1\lor\frac{(eZ^2)^{1/d}}{2\pi}\right)^{\tfrac{\alpha}{2-\alpha}}\right\} \, \lor \, \left(\frac{2^{\tfrac{2}{2-\alpha}}eb\pi(\norm{\cdot}^{\tfrac{2q}{q-1}})^{\tfrac{q-1}{q}}}{\alpha^{2/\alpha-1}d}\right)^{\tfrac{\alpha}{2-\alpha}}\lor\frac{1}{\alpha}.
        \end{equation}
        This time, we choose $\rho_0 = \mathcal{N}(0,\sigma^2 I_d)$ with $\sigma^2 = \frac{(\alpha \Delta_0 )^{\tfrac{2-\alpha}{\alpha}}}{b}$. Then, \eqref{eq:delta0_lower_sublin} ensures $\sigma^2 \geq 1$ and by Lemma \ref{lem:init_renyi},
        $$R_\infty(\rho_0\mmid\pi) \leq \Delta_0 + \ln\left(\frac{Z}{(2\pi\sigma^2)^{d/2}}\right) + \frac{1}{2\sigma^2} \leq \Delta_0.$$
        From Lemma \ref{lem:lang_gf_comparison}
        $$T \geq \frac{\Earg{\norm{X_0}^2}^{1-\alpha/2} - \Earg{\norm{X_T}^2}^{1-\alpha/2}}{b(2-\alpha)},$$
        hence we can write
        \begin{align*}
            T &\geq \frac{\left(\tfrac{\alpha^{2/\alpha-1}}{b}\right)^{1-\alpha/2}d^{1-\alpha/2}\Delta_0^{\tfrac{(2-\alpha)^2}{2\alpha}} - (e\pi(\norm{\cdot}^{\tfrac{2q}{q-1}})^{\tfrac{q-1}{q}})^{1-\alpha/2}}{b(2-\alpha)}\\
            &\geq \frac{\left(\tfrac{\alpha^{2/\alpha-1}}{b}\right)^{1-\alpha/2}d^{1-\alpha/2}\Delta_0^{\tfrac{(2-\alpha)^2}{2\alpha}}}{2(2-\alpha)b}.
        \end{align*}
        For the discrete-time case, from Lemma \ref{lem:lang_gf_comparison} we obtain
        $$\Earg{\norm{x_{k+1}}^2} \geq \Earg{\norm{x_k}^2} - 2hb\Earg{\norm{x_k}^2}^{\alpha/2}.$$
        Let $r_k \coloneqq \Earg{\norm{x_k}^2}$. Suppose $r_k \geq r_{k+1}$, rearranging the inequality above, we obtain
        $$2hb \geq r_k^{1-\alpha/2} - r_{k+1}r_{k}^{-\alpha/2} \geq r_k^{1-\alpha/2} - r_{k+1}^{1-\alpha/2}.$$
        On the other hand, when $r_k < r_{k+1}$,
        $$2hb > 0 > r_k^{1-\alpha/2} - r_{k+1}^{1-\alpha/2}.$$
        Thus the bound holds in either case, and by iterating it we have
        \begin{align*}
            N \geq \frac{\Earg{\norm{X_0}^2}^{1-\alpha/2} - \Earg{\norm{X_T}^2}^{1-\alpha/2}}{hb(2-\alpha)}\geq \frac{\left(\tfrac{\alpha^{2/\alpha-1}}{\tilde{a}}\right)^{1-\alpha/2}d^{1-\alpha/2}\Delta_0^{\tfrac{(2-\alpha)^2}{2\alpha}}}{2(2-\alpha)hb},
        \end{align*}
        where the second inequality follows analogously to the continuous-time case.
        \item \underline{The case $\alpha = 2$}: Suppose that $\Delta_0$ satisfies
        \begin{equation}\label{eq:delta0_lower_gaussian}
            \Delta_0 \geq \frac{bZ^{2/d}e^{2a/d-1}}{4\pi} \lor b\left(e\pi(\norm{\cdot}^{\tfrac{2q}{q-1}})\right)^{\tfrac{(1 + c)(q-1)}{q}}.
        \end{equation}
        for any absolute constant $c > 0$. Choose $\rho_0 = \mathcal{N}(0,\sigma^2 I_d)$ with $\sigma^2 = \tfrac{2\Delta_0}{bd}$.
        Then, by Lemma \ref{lem:init_kl} we have
        $$\KL(\rho_0\mmid\pi) \leq \Delta_0 -\frac{d}{2} + \ln\left(\frac{Z}{(2\pi\sigma^2)^{d/2}}\right) \leq \Delta_0.$$
        Moreover, from Lemma \ref{lem:lang_gf_comparison}, we have
        $$T \geq \frac{\ln\left(\Earg{\norm{X_0}^2}\right) - \ln\left(\Earg{\norm{X_T}^2}\right)}{2b} \geq \frac{c\ln\left(\frac{\Delta_0}{b}\right)}{2(1+c)b}.$$
        Similarly for LMC, when $h < b^{-1}$, we have
        $$N \geq \frac{\ln\left(\Earg{\norm{x_0}^2}\right) - \ln\left(\Earg{\norm{x_N}^2}\right)}{2hb} \geq \frac{c\ln\left(\frac{\Delta_0}{b}\right)}{2(1+c)b},$$
        which completes the proof of the theorem.
    \end{enumerate}
\end{proof}

In order to prove Proposition \ref{prop:step_size_bound_summary}, we need a sharper control on the decay of the second moment of LMC that does not ignore terms of order $O(h^2)$. In the following lemma, we achieve this control in the radially symmetric setting.
\begin{lemma}[A Sharper Evolution Inequality for LMC]\label{lem:lmc_gd_comparison}
Suppose \(\Earg{\norm{x_0}^2} < \infty\), the potential is radially symmetric with $V(x) = f(\norm{x}^2)$ and the function \(g: \R_+ \to \R_+\) given by \(g(r) = (1 - 2 h f'(r))^2 r\) is convex, then for each \(k \geq 0\),
\begin{equation*}
    \E\norm{x_{k+1}}^2 \geq g(\Earg{\norm{x_k}^2}) + 2hd.
\end{equation*}
Furthermore, if \(g\) is non-decreasing then, \(\E \|x_k\|^2 \geq \|y_k\|^2\) for each \(k \geq 0\), where we define \(y_k\) by
\begin{equation*}
    y_{k+1} = y_k - \eta \nabla V(y_k), \quad \|y_0\|^2 = \Earg{\norm{x_0}^2}.
\end{equation*}
\end{lemma}
\begin{proof}
Using the independence of the Gaussian perturbations and the fact that \(\nabla V(x) = 2 f'(\|x\|^2) x\),
\begin{equation*}
    \Earg{\norm{x_{k+1}}^2} = \E\norm{x_k - h\nabla V(x_k)}^2 + 2 h d = \E g(\|x_k\|^2) + 2hd.
\end{equation*}
Using the convexity of \(g\) along with Jensen's inequality, we conclude that
\begin{equation*}
    \Earg{\norm{x_{k+1}}^2} \geq g(\Earg{\norm{x_k}^2}) + 2hd.
\end{equation*}
If \(g\) is non-decreasing, it follows by comparison that \(\Earg{\norm{x_k}^2} \geq \|y_k\|^2\).
\end{proof}

Finally, we can present the proof of Proposition \ref{prop:step_size_bound_summary}.

\begin{proof}[Proof of Proposition \ref{prop:step_size_bound_summary}]
First, notice that by Lemma \ref{lem:wass_comparison}, $\inf_k R_q(\mu_k\mmid\pi) < \varepsilon$ is equivalent to \(\inf_{k \in \N} \Earg{\|x_k\|^2} < \sigma^2_\varepsilon\). Let \(z_k\) be the process defined by the update
\begin{equation*}
    z_{k+1} = g(z_k) + 2hd, \quad z_0 = \Earg{\norm{x_0}^2}
\end{equation*}
so that, using Lemma \ref{lem:lmc_gd_comparison}, we have \(\Earg{\norm{x_k}^2} \geq z_k\). Thus, if it holds that \(\inf_{k \in \N} \Earg{\|x_k\|^2} < \sigma^2_\varepsilon\) then \(\inf_{k \in \N} z_k < \sigma^2_\varepsilon\) must hold also. If this holds, there must be some \(k \in \N\) such that \(z_{k+1} \leq \sigma^2_\varepsilon\) and \(z_k \geq \sigma^2_\varepsilon\). Thus, by the fact that \(g\) is non-decreasing,
\begin{equation*}
    g(\sigma_\varepsilon^2) + 2hd \leq g(z_k) + 2hd = z_{k+1} \leq \sigma_\varepsilon^2.
\end{equation*}
Rearranging this leads to the bound given in the statement.
\end{proof}

%% file: appendices/auxiliary.tex
In this section, we prove various moment and tail bounds for generalized Cauchy measures and measures with sublinear potentials, which we use in other proofs of the paper.
\begin{lemma}\label{lem:subexp_moments}
Consider the measure $\tilde{\pi}_\alpha(x) \propto \exp(-\lambda \norm{x}^\alpha)$ for $0 < \alpha \leq d$ and \(\lambda > 0\). Then, for any $p > 0$
\begin{equation}
    \E_{\tilde{\pi}_\alpha}\left[\norm{x}^p\right] = \lambda^{-p/\alpha} \frac{\Gamma\left(\frac{d + p}{\alpha}\right)}{\Gamma\left(\frac{d}{\alpha}\right)} \leq \lambda^{-p/\alpha} \left(\frac{d+p}{\alpha}\right)^{\tfrac{p}{\alpha}}.
\end{equation}
Moreover, for $\alpha \in (0, 1)$ and $\pi_\alpha(x) \propto \exp(-(1+\lambda^{2/\alpha}\norm{x}^2)^{\alpha/2})$,
\begin{equation}
    \E_{\pi_\alpha}\left[\norm{x}^p\right] \leq \lambda^{-p/\alpha} \frac{e\Gamma\left(\frac{d + p}{\alpha}\right)}{\Gamma\left(\frac{d}{\alpha}\right)}.
\end{equation}
\end{lemma}
\begin{proof}
First, consider the case of \(\lambda = 1\). We begin by computing the normalizing factor of $\tilde{\pi}_\alpha$. Using the polar coordinates, we have
\begin{align}
    Z \coloneqq \int\exp(-\norm{x}^\alpha)\D x &= d\omega_d\int_0^\infty \exp(-r^\alpha)r^{d-1}\D r\nonumber\\
    &=\frac{d\omega_d}{\alpha}\int_0^\infty\exp(-u)u^{d/\alpha - 1}\D u\nonumber\\
    &= \frac{d\omega_d}{\alpha}\Gamma\left(\frac{d}{\alpha}\right).\label{eq:sublin_normalizing_const}
\end{align}
Via similar calculations, we obtain
$$\int\exp(-\norm{x}^\alpha)\norm{x}^p\D x = \frac{d\omega_d}{\alpha}\Gamma\left(\frac{d+p}{\alpha}\right).$$
Applying Lemma \ref{lem:gamma_ratio} yields,
$$\E_{\tilde{\pi}_\alpha}\left[\norm{x}^p\right] = \frac{\Gamma\left(\frac{d+p}{\alpha}\right)}{\Gamma\left(\frac{d}{\alpha}\right)} \leq \left(\frac{d+p}{\alpha}\right)^{\tfrac{p}{\alpha}}.$$
Finally, we observe that for $\pi_\alpha(x) \propto (1 + \norm{x}^2)^{\alpha/2}$,
$$\E_{\pi_\alpha}\left[\norm{x}^p\right] = \frac{\int\exp\left(-(1+\norm{x}^2)^{\tfrac{\alpha}{2}}\right)\norm{x}^p\D x}{\int\exp\left(-(1+\norm{x}^2)^{\tfrac{\alpha}{2}}\right)\D x} \leq \frac{\int\exp(-\norm{x}^\alpha)\norm{x}^p\D x}{e^{-1}\int\exp(-\norm{x}^\alpha)\D x} = \frac{e\Gamma\left(\frac{d+p}{\alpha}\right)}{\Gamma\left(\frac{d}{\alpha}\right)}.$$
For the case of \(\lambda > 0\), we use the change of variables formula to show that scaling by \(\lambda^{-1/\alpha}\) recovers a random variable with the density given by the case with \(\lambda = 1\).
\end{proof}

\begin{lemma}\label{lem:student_t_moments}
    Consider the measure $\pi_\nu(x) \propto (1 + \norm{x}^2)^{-(d+\nu)/2}$ for $\alpha > 0$ with $\nu > p \geq 0$. Then,
    $$\E_{\pi_\nu}\left[\norm{x}^p\right] = \frac{d}{d+p}\frac{\Gamma\left(\frac{\nu-p}{2}\right)}{\Gamma\left(\frac{\nu}{2}\right)}\frac{\Gamma\left(\frac{d+2 + p}{2}\right)}{\Gamma\left(\frac{d+2}{2}\right)} \leq \frac{\Gamma\left(\frac{\nu-p}{2}\right)}{\Gamma\left(\frac{\nu}{2}\right)}\left(\frac{d + 2 + p}{2}\right)^{p/2}.$$
\end{lemma}
\begin{proof}
    Recall that the normalizing constant of this measure is given by
    $$Z_{d,\nu} \coloneqq \frac{\Gamma\left(\tfrac{\nu}{2}\right)\pi^{d/2}}{\Gamma\left(\tfrac{\nu+d}{2}\right)} = \frac{\Gamma\left(\tfrac{\nu}{2}\right)\Gamma\left(\frac{d+2}{2}\right)\omega_d}{\Gamma\left(\tfrac{\nu+d}{2}\right)}.$$
    On the other hand, using the polar coordinates, one can observe
    \begin{equation}\label{eq:cauchy_z_int}
        Z_{d,\nu} = d\omega_d\int (1 + r^2)^{-(\nu+d)/2}r^{d-1}\D r.
    \end{equation}
    We proceed to compute the following
    $$\E_{\pi_\nu}\left[\norm{x}^p\right] = \frac{d\omega_d}{Z_{d,\nu}}\int (1+r^2)^{-(d+\nu)/2}r^{d+p-1}\D r = \frac{d\omega_d}{Z_{d,\nu}}\frac{Z_{d+p,\nu-p}}{(d+p)\omega_{d+p}},$$
    where the second equality follows from a change of variables in \eqref{eq:cauchy_z_int}. The statement of the lemma follows by an application of Lemma \ref{lem:gamma_ratio}.
\end{proof}

\begin{lemma}\label{lem:tail bound subexp} The measure $\pi_\alpha\propto \exp\left( -\left(1+\norm{x}^2\right)^{\alpha/2}\right)$ with $\alpha\in (0,1)$ satisfies
    \begin{align}\label{eq:tail bound subexp}
        \pi_\alpha\left( \norm{x}\ge R \right)\le e^{\frac{1}{2}} 2^{d/\alpha} \exp\left(-\frac{1}{2}\left(1+R^2\right)^{\alpha/2}\right) .
    \end{align}
\end{lemma}
\begin{proof}
    Using the Markov inequality,
    \begin{align*}
        \pi_\alpha\left( \norm{x}\ge R \right)&= \pi_\alpha\left\{\exp\left(\frac{1}{2}\left(1+\norm{x}^2\right)^{\alpha/2}\right)\ge \exp\left(\frac{1}{2}\left(1+R^2\right)^{\alpha/2}\right)  \right\}\\
        &\le \exp\left(-\frac{1}{2}\left(1+R^2\right)^{\alpha/2}\right) \mathbb{E}_{\pi_\alpha} \left[ \exp\left(\frac{1}{2}\left(1+\norm{x}^2\right)^{\alpha/2}\right) \right]. 
    \end{align*}
Using polar coordinates and change of coordinates,
\begin{align*}
    \mathbb{E}_{\pi_\alpha} \left[ \exp\left(\frac{1}{2}\left(1+\norm{x}^2\right)^{\alpha/2}\right) \right]&=\frac{\int \exp\left(-\frac{1}{2}\left(1+\norm{x}^2\right)^{\alpha/2}\right) \D x}{\int \exp\left(-\left(1+\norm{x}^2\right)^{\alpha/2}\right) \D x}\\
    &=\frac{\int_0^\infty r^{d-1}\exp\left( -\frac{1}{2}\left( 1+r^2\right)^{\alpha/2} \right)\D r}{\int_0^\infty r^{d-1}\exp\left( -\left( 1+r^2\right)^{\alpha/2} \right)\D r}\\
    &=\frac {
    e^{-\frac{1}{2}} 
    \int_0^\infty \left(u+1\right)^{2/\alpha-1}\left[\left(u+1\right)^{2/\alpha}-1\right]^{(d-2)/2}\exp\left(-u/2\right)\D u
    }{
    e^{-1}
    \int_0^\infty \left( u+1 \right)^{2/\alpha-1}\left[\left(u+1\right)^{2/\alpha}-1\right]^{(d-2)/2}\exp\left(-u\right)\D u
    }\\
    &\le e^{\frac{1}{2}} 2^{d/\alpha},
\end{align*}
where we used the change of variables $u = (1 + r^2)^{\alpha/2} - 1$. This completes the proof.
\end{proof}

\begin{lemma}\label{lem:student_t_tail}
    The measure $\pi_\nu(x) \propto (1 + \norm{x}^2)^{-(d+\nu)/2}$ satsifies
    \begin{equation}
        \pi_\nu(\norm{x} \geq R) \leq (\nu+d)^{\nu/2}R^{-\nu}.
    \end{equation}
\end{lemma}
\begin{proof}
    Using polar coordinates
    $$\pi_\nu(\norm{x} \geq R) = \frac{1}{Z}\int_{\norm{x} \geq R}(1+\norm{x}^2)^{-(d+\nu)/2}\D x = \frac{d\omega_d}{Z}\int_R^\infty(1+r^2)^{-(d+\nu)/2}r^{d-1}\D r \leq \frac{d\omega_d R^{-\nu}}{\nu Z}.$$
    And $Z = \tfrac{\Gamma(\nu/2)\pi^{d/2}}{\Gamma((\nu+d)/2)}$. Hence,
    $$\pi_\nu(\norm{x} \geq R) \leq \underbrace{\frac{d\Gamma\left(\frac{\nu+d}{2}\right)}{\nu\Gamma\left(\frac{\nu}{2}\right)\Gamma\left(\frac{d+2}{2}\right)}}_{\eqqcolon A_{\nu,d}}R^{-\nu}.$$
    Suppose $\nu < 2$. Then by Equation (3.1) of \cite{laforgia2013some}, we have $\tfrac{\Gamma((d+\nu)/2)}{\Gamma((d+2)/2)} \leq (d/2)^{\nu/2-1}$. Moreover, when $\nu \geq 2$, using Lemma \ref{lem:gamma_ratio} we have $\tfrac{\Gamma((d+\nu)/2)}{\Gamma((d+2)/2)} \leq ((d+\nu)/2)^{\nu/2-1}$. Consequeuntly,
    $$A_{\nu,d} = \frac{d^{\nu/2}(1/2)^{\nu/2-1}(1\lor(1+\nu/d)^{\nu/2-1})}{\nu\Gamma(\nu/2)} \leq \frac{(d + \nu)^{\nu/2}}{2^{\nu/2}\Gamma(\nu/2 + 1)} \leq (d + \nu)^{\nu/2},$$
    where we used the fact that $2^{\nu/2}\Gamma(\nu/2 + 1) \geq \Gamma(1) = 1$.
\end{proof}

The following Lemma, adapted from \cite{Chewi2021-vj}, shows the existence of isotropic Gaussian initializations such that $R_q(\mu_0\mmid\pi),R_q(\mu_0\mmid\hat{\pi}) = \tilde{O}(d)$.
\begin{lemma}\label{lem:init}
    Let $\pi(x) \propto \exp(-V(x))$ such that $\nabla V$ is s-H\"older continuous and $\nabla V(0) = 0$. Define $\hat{\mu}$ as in Proposition \ref{prop:disc_error}. Let $\mu_0 = \mathcal{N}(0,(2L + 1)^{-1}I_d)$. Then,
    \begin{align}
        R_\infty(\mu_0\mmid\pi) &\leq 2 + L + V(0) - \min_x V(x) + \frac{d}{2}\ln(12m^2L),\\
        R_\infty(\mu_0\mmid\hat{\pi}) &\leq 3 + L + V(0) - \min_x V(x) + \frac{d}{2}\ln(12(m + 6144T)^2L).
    \end{align}
\end{lemma}
\begin{proof}
    The Lemma is directly based on Lemmas 30 and 31 of \cite{Chewi2021-vj}.
\end{proof}

The following lemma translates a bound on $R_2(\mu_0\mmid\pi)$ to a bound on $R_2(\mu_0\mmid\hat{\pi})$ when $\mu_0$ is some isotropic Gaussian measure. As we only calculate the former quantity for our model examples, we use this lemma to establish similar bounds for the latter.
\begin{lemma}\label{lem:initial_renyi_alt}
    Suppose $\pi \propto \exp(-V(x))$ and $\hat{\pi} \propto \exp\left(-\hat{V}(x)\right)$ with
    $$\hat{V}(x) = V(x) + \frac{\gamma}{2}\max\{\norm{x} - R, 0\}^2$$
    for some $\gamma,R > 0$. Then, for any $\sigma^2 \leq \tfrac{1}{\gamma}$ we have
    $$R_2(\mathcal{N}_{\sigma^2 I_d}\mmid\hat{\pi}) \leq d\ln2 + R_2(\mathcal{N}_{2\sigma^2 I_d}\mmid\pi).$$
\end{lemma}
\begin{proof}
    Let $Z \coloneqq \int\exp(-V(x))\D x$ and $\hat{Z} \coloneqq \int \exp(-\hat{V}(x))\D x$. Notice that $V \leq \hat{V}$, thus $\hat{Z} \leq Z$. Therefore,
    \begin{align*}
        R_2(\mathcal{N}_{\sigma^2 I_d}\mmid\hat{\pi}) &= \ln\left(\frac{\hat{Z}}{(2\pi\sigma^2)^d}\int\exp\left(\frac{-\norm{x}^2}{\sigma^2} + V(x) + \frac{\gamma}{2}\max\{\norm{x}-R,0\}^2\right)\D x\right)\\
        &\leq \ln\left(\frac{Z}{(2\pi\sigma^2)^d}\int\exp\left(\frac{-\norm{x}^2}{2\sigma^2} + V(x)\right)\D x\right)\\
        &= d\ln2 + R_2(\mathcal{N}_{2\sigma^2 I_d}\mmid\pi).
    \end{align*}
\end{proof}

\begin{lemma}[{\cite[Theorem 3.1]{laforgia2013some}}]\label{lem:gamma_ratio}
    Suppose $x \geq y \geq 1$, then
    \begin{equation}
        y^{x-y}e^{1-\tfrac{x}{y}} \leq \frac{\Gamma(x)}{\Gamma(y)} \leq x^{x-y}.
    \end{equation}
\end{lemma}

\begin{lemma}\label{lem:raw_central_moment_compare}
    Suppose $Z \geq 0$ is a non-negative random variable. Then, for any $p \geq 2$,
    $$\Earg{Z^p} \geq \Earg{\abs{Z-\Earg{Z}}^p}.$$
\end{lemma}
\begin{proof}
    Normalize $Z$ such that $\Earg{Z} = 1$. Using the inequality $Z^p - \abs{Z-1}^p \geq Z-1$ for every $Z \geq 0$ and $p \geq 2$ and taking expectations proves the lemma.
\end{proof}